\documentclass[reqno]{amsart}

\usepackage{mathtools}
\usepackage[T1]{fontenc}
\usepackage{lmodern}
\usepackage{enumitem}
\usepackage{hyperref}
\hypersetup{colorlinks=true,linkcolor=blue,citecolor=blue,urlcolor=blue}

\UseRawInputEncoding
\usepackage{tikz}
\usepackage{amsmath, amsfonts,amssymb,amsthm,amscd,latexsym,cite}
\usepackage{mathrsfs}
\usepackage{enumitem}
\usepackage{color}
\usepackage{verbatim}
\newtheorem{thm}{Theorem}[section]
\newtheorem{theorem}[thm]{Theorem}

\newtheorem{lemma}[thm]{Lemma}

\newtheorem{corollary}[thm]{Corollary}
\newtheorem{proposition}[thm]{Proposition}

\newtheorem{remark}[thm]{Remark}

\newcommand{\cA}{{\mathcal A}}
\newcommand{\cB}{{\mathcal B}}

\newcommand{\cD}{{\mathcal D}}

\newcommand{\cF}{{\mathcal F}}

\newcommand{\cH}{{\mathcal H}}

\newcommand{\cM}{{\mathcal M}}
\newcommand{\cN}{{\mathcal N}}

\newcommand{\cP}{{\mathcal P}}

\newcommand{\tx}{\tilde{x}}
\newcommand{\tM}{\tilde{\mathcal{M}}}
\newcommand{\ttau}{\tilde{\tau}}

\setlist[enumerate,1]{label=\textup{(\roman*)}}

 \usepackage{color}
  \newcommand{\Tr}{\mathrm{Tr}}
    
 \newcommand{\norm}[1]{\left\lVert#1\right\rVert}
  \newcommand{\cNorm}[1]{\left\lVert#1\right\rVert}
\usepackage{hyperref}					
\hypersetup{colorlinks,
	linkcolor=blue,%
	citecolor=blue}

 \usepackage{float}
\begin{document}

\title[Jensen' s trace inequality with equality condition and noncommutative Lamperti's theorem]{Jensen' s trace inequality with equality condition and noncommutative Lamperti's theorem}

\author[K. Fang]{Kai Fang}
\author[X. He]{Xin He}
\author[J. Huang]{Jinghao Huang}

\address{Institute for  Advanced Study in  Mathematics of HIT, Harbin Institute of Technology, Harbin, 150001, China}
\email{{\color{blue}kaifang.8.25@gmail.com; jinghao.huang@hit.edu.cn}}

\address{Department of Mathematics, Harbin Normal University, Harbin, 150001, China}
\email{{\color{blue}hexin8323@163.com}}

\thanks{K. Fang and J. Huang were supported the NNSF of China (No. 12031004, 12301160, 12471134 and 12671159);
X. He was supported the NNSF of China (No. 12471129)}

\subjclass[2020]{46B04; 46E30; 46L10; 47A63}
	\keywords{Jensen's trace inequality; Orlicz space; isometry.}

\begin{abstract}

Let $\mathcal M$ be a semifinite von Neumann algebra equipped with a semifinite faithful normal trace $\tau$. We establish Jensen's trace inequality in full generality and characterize its equality case, 
which answers  two questions raised in 
[Kosaki2013] and [HaradaKosaki2008]. 
As an application, we derive
noncommutative Lamperti-type inequalities and their
equality conditions. 
Employing this result, we characterize linear isometries  (not necessarily surjective) on a class of
$F$-normed noncommutative  Orlicz spaces, which
provides a noncommutative Lamperti's
theorem for linear isometries. 

\end{abstract}
\maketitle
\section{Introduction}

\subsection{Equality condition  for Jensen's trace inequality}

Jensen's inequality, originating from the classical theory of convex 
functions \cite{Jensen1906}, has played a fundamental role in matrix 
analysis and operator theory. Its noncommutative extension is closely 
related to the theory of operator convex functions. Davis 
\cite{Davis1957} characterized operator convexity by Jensen's inequality under orthogonal compressions. 
Choi \cite{Choi1974} extended Davis's compression inequality to
unital positive linear mappings between $C^{*}$-algebras, establishing
the corresponding Jensen-type inequality for operator convex functions.
Pedersen \cite{HansenPedersen1982} further developed Jensen's operator 
inequality and showed that, for an operator convex function $f$,
\[
 f\left(\sum_{i=1}^{n} a_i^{*}x_i a_i\right)
 \leq
 \sum_{i=1}^{n} a_i^{*}f(x_i)a_i,
 \quad
 \sum_{i=1}^{n}a_i^{*}a_i=\mathbf{1}.
\]

A significant feature of the trace version of Jensen's inequality is 
that operator convexity can often be replaced by ordinary convexity. 
In the setting of semifinite von Neumann algebras equipped with a semifinite faithful normal trace $\tau$, Brown and Kosaki 
\cite{BrownKosaki} proved that
\begin{align}\label{eq:Jensen's trace inequality}
 \tau\left(f(a^{*}xa)\right)
 \leq
 \tau\left(a^{*}f(x)a\right)
\end{align}
for suitable convex functions $f$, positive operators $x$ (not necessarily bounded), and 
contractions $a$ (i.e., $\norm{a}_\infty\le 1$) whenever both sides are well-defined. 
Petz \cite{Petz1987} studied a related Jensen's trace 
inequality for positive contractive linear mappings between operator algebras.
Hansen and Pedersen \cite{HansenPedersen2003} subsequently formulated
the Jensen's trace inequality for noncommutative convex combinations in
$C^{*}$-algebras with a finite trace.
Harada and Kosaki \cite{HaradaKosaki2010} established Jensen's trace inequality for semibounded self-adjoint $\tau$-measurable operators and observed that the semiboundedness assumption can be removed in certain special cases. Motivated by this, Kosaki \cite{Kosaki2013} subsequently posed and investigated the question (see \cite[p. 2]{Kosaki2013})
\begin{quote} whether the trace inequality \eqref{eq:Jensen's trace inequality} is
valid for a self-adjoint $\tau$-measurable operator $x$ under suitable integrability requirement (guaranteeing the well-definedness of the above both sides).
\end{quote}
However, Kosaki \cite{Kosaki2013}
  obtained affirmative results only under additional assumptions. Thus, the problem remains unresolved in full generality.
The recent results  
\cite{CarlenFrankLarson2025, RahamanTurowska2026} 
concern other extensions of Jensen's inequality and did not solve this problem.

The equality problem has also received considerable attention. For the operator Jensen's inequality, Petz
\cite{Petz1986} proved that equality for a non-affine operator convex
function at a self-adjoint element $x$ holds if and only if the
underlying unital positive mapping is multiplicative on the unital
$C^{*}$-subalgebra $C^{*}(x,1)$ generated by $x$.
In the 
trace setting, Harada and Kosaki \cite{HaradaKosaki} showed, under 
appropriate regularity and structural assumptions on $f$, 
\[
 \tau\left(f(a^{*}xa)\right)
 =
 \tau\left(a^{*}f(x)a\right)
 \Longleftrightarrow
  a^{*}x^{2}a=(a^{*}xa)^{2}.
\]
Nevertheless, it remains  open (see \cite[p. 482]{HaradaKosaki}) 
    \begin{quote} whether this equality characterization remains valid under ordinary strict convexity beyond the cases of matrices and compact operators.
    \end{quote}

In the present paper, we establish Jensen's trace inequality in full generality and characterize its equality case,
which answers the questions raised in \cite{Kosaki2013, HaradaKosaki} in the affirmative.
\begin{theorem}\label{equality condition}
Assume that $\cM$ is a  semifinite von Neumann algebra equipped with a semifinite faithful normal trace $\tau$.
Let $a\in \cM$ be a contraction and let $x\in S(\cM,\tau)_h$.
Suppose that $I$ is a non-degenerate interval containing $0$ and $\sigma(x),\sigma(a^*xa)$.
If $f:I\to\mathbb R$ is convex and continuous
such that $f(0)=0$ and $a^*f(x)a,f(a^*xa)\in L_1(\cM,\tau)$,
then
\[
\tau(a^*f(x)a)\ge \tau(f(a^*xa)).
\]
In the case where $f$ is strictly convex, equality holds if and only if $a^*x^2a=(a^*xa)^2$.
\end{theorem}

To prove this theorem, we first establish the inequality and  its equality condition in the special case
where $a=e$ is a $\tau$-finite projection (see Proposition~\ref{prop:finite} below), and then extend the result
to the general case. 
A key ingredient is a new scalar representation
in terms of a probability kernel, established in Lemma~\ref{lem:kernel}. 
In contrast to the results of Harada and Kosaki \cite{HaradaKosaki2010} and Kosaki \cite{Kosaki2013}, which require additional assumptions, 
this representation enables us to establish Jensen's trace inequality in full generality.
Moreover, this representation reduces the trace equality to the equality case of the classical scalar Jensen inequality, which allows us to prove 
the equality condition when \(f\) is strictly convex, without requiring \(f\) to have the composite form in Harada and Kosaki \cite{HaradaKosaki}.

As further consequences of Theorem \ref{equality condition}, Corollary~\ref{prop:two contractions} extends, in the setting of
self-adjoint $\tau$-measurable operators, the Jensen's trace inequality for
noncommutative convex combinations of Hansen and Pedersen \cite{HansenPedersen2003}, and provides
its equality characterization under ordinary strict convexity or strict
concavity. 
Moreover, Proposition~\ref{prop:eq con submajor} and its direct consequence, Proposition~\ref{prop:rk-equality}, provide the condition for equalities to hold in \cite[Proposition 4.6]{FackKosaki},
\cite[Proposition 15 and Proposition 16]{BrownKosaki} and \cite[Theorem 5.3]{DS},
which extends the result in ~\cite[Remark~11]{HaradaKosaki}.

\subsection{Noncommutative Lamperti-type results}

The equality cases of Lamperti-type inequalities (see \cite[Theorem 2.1]{Lamperti}) provide an important
tool for studying isometries of function spaces. 
More precisely,
Lamperti proved, under suitable strict convexity or
concavity assumptions on $t\mapsto\Phi\left(\sqrt{t}\right)$, that equality in
the corresponding Lamperti-type inequality holds if and only if the
functions involved are disjoint. 
This implies that every linear
modular-preserving mapping $T$ on a certain modular class preserves disjointness and consequently is of the elementary form
\[
(Tf)(t)=h(t)(T_1f)(t),
\]
where $h$ is a measurable function and $T_1$ is induced by
a regular set isomorphism (see \cite[Theorem~4.1]{Lamperti}).
For \(\Phi(t)=t^{p}\), 
the modular class is $L_p$ space and the result above provides  characterization of isometries on $L_p$ spaces (see \cite[Theorem 3.1]{Lamperti}).

In the noncommutative \(L^{p}\)-setting, Yeadon \cite{Yeadon}
gave the equality conditions for the noncommutative Clarkson's inequality and thereby obtained a characterization of isometries on noncommutative $L_p$ spaces.

In the present paper, we extend this approach to a certain class of  $F$-normed noncommutative
 Orlicz spaces $L_\Phi(\cM,\tau)$. Using 
Jensen's trace
inequality and its 
  equality condition  established above, we derive noncommutative Lamperti-type
inequalities and characterize their equality cases in terms of
disjointness.
\begin{theorem}\label{traceineq}\label{thm:lamperti}
Let $x,y\in S(\cM,\tau)$, and let
$\Phi:[0,\infty)\to[0,\infty)$ be continuous and strictly increasing with
$\Phi(0)=0$.  Denote
\[
        \psi(t):=\Phi\left(\sqrt t\right),
        \quad t\ge0.
\]
Assume all traces below are finite.
\begin{enumerate}[label={\rm(\alph*)}]
\item If $\psi$ is convex, then
\[
        \tau(\Phi(|x+y|))+\tau(\Phi(|x-y|))
        \ge
        2\tau(\Phi(|x|))+2\tau(\Phi(|y|)).
\]
In the case where $\psi$ is strictly convex,  equality holds if and only if $x^*y=0$ and $xy^*=0$.
\item If $\psi$ is concave, then the reverse inequality holds.
In the case where $\psi$ is strictly concave, equality holds if and only if $x^*y=0$ and $xy^*=0$.
\end{enumerate}
\end{theorem}
Within the semifinite setting,
this extends 
equality conditions for the noncommutative Clarkson's inequality
(see \cite[Theorem 1]{Yeadon} and \cite[Theorem A.1]{RX}).
Having this equality condition at hand, 
we obtain   a   noncommutative version of Lamperti's result \cite[Theorem 4.1]{Lamperti}, showing that 
 every isometry (not necessarily surjective) $T$ on a certain $F$-normed noncommutative
 Orlicz spaces $L_\Phi(\cM,\tau)$ is disjointness-preserving and has the   elementary form.
\begin{theorem}\label{Tform}
    Assume that $\Phi$ on $[0,\infty)$  is continuous and strictly increasing with $\Phi(0)=0$, $\Phi(1)=1$.
    Let $\cM$ and $\cN$ be two semifinite von Neumann algebras equipped with semifinite faithful normal traces $\tau$ and $\nu$, respectively.
    Suppose that $\Phi\left(\sqrt{t}\right)$ is either strictly convex or strictly concave. 
If $T:L_\Phi(\cM,\tau)\to L_\Phi(\cN,\nu)$ is an isometry,
then there exist a partial isometry $u\in\cN$, 
a Jordan $*$-monomorphism $J:\cM\to \cN$
and a positive operator $b$ (possibly not measurable) affiliated with $Z(J(\cM))$
such that
\[
T(x)=ubJ(x),\quad x\in L_\Phi(\cM,\tau)\cap \cM
\]
and $u^*u=s(b)=J(\mathbf{1})$.
Moreover,
\[
\tau(x)=\nu(\Phi(b)J(x)), \quad 0\le x\in L_1(\cM,\tau)\cap \cM.
\]
\end{theorem}

\section{Preliminaries}

In this section,
we recall some basic facts and notions which are needed for the proofs of the main results of this paper. 

\subsection{$\tau$-measurable operators,  singular value function and measure topology}


Let $\mathcal{M}$ be a   von Neumann algebra on a   Hilbert space $\cH$.
Let ${\bf 1}$ be the identity.
Let $P(\mathcal{M})$ denote the lattice of all projections in $\mathcal{M}$,
$U(\cM)$ denote the set of all unitary elements in $\cM$
 and $Z(\cM)$ denote the center of $\cM$.
The set of all self-adjoint elements in $\cM$ is denoted by $\cM_h$ and the set of all positive elements in $\cM$ is denoted by $\cM_+$.
For each self-adjoint operator $x$ affiliated with $\cM$,
we denote its spectral measure by $\{e^x\}$.
We say that a linear operator  $x$ is {\it measurable} (denoted by $x \in S(\cM)$) if and only if $x$ is closed,
densely defined, affiliated with $\cM$, and $e 
^{|x|}
(\lambda,\infty )$ is a finite projection in $\cM$ for some $\lambda >0$. It follows
immediately that in the case when $\cM$ is a von Neumann algebra of type $III$ or a type $I$ factor,
we have $S(\cM) =\cM$. 
For type $II$ von Neumann algebras, this is no longer true\cite{DPS,LSZ}.

Let $\cM$ be a semifinite von Neumann algebra on a  Hilbert space $\cH$ equipped with a faithful normal semifinite trace $\tau$.
A measurable operator $x $ affiliated with $\cM$ is called {\it $\tau$-measurable} if
$\tau(e ^{|x|}(\lambda,\infty))<\infty$ for sufficiently large $\lambda$.
We denote the set of all $\tau$-measurable operators by
$S(\mathcal{M},\tau)$, which is a unital $^*$-algebra with respect to
strong sums and products (denoted simply by $x+y$ and $xy$ for all $x,y\in S(\cM,\tau)$)\cite{DPS,LSZ}.
The set of all self-adjoint elements in $S(\cM,\tau)$ is denoted by $S(\cM,\tau)_h$ and the set of all 
positive elements in $S(\cM,\tau)_h$ is denoted by $S(\cM,\tau)_+$.

For any closed and densely defined linear operator $x$, the null projection $n(x)=n(|x|)$ is the projection onto its kernel Ker$(x)$.
The support projection $s(x)$ of $x$ is defined by $s(x)=\mathbf{1}-n(x)$, which is the projection onto the closure of Ran$(x^*)$.

The two-sided ideal $\mathcal{F}(\cM,\tau)$ in $\cM$ consisting of all elements with $\tau$- finite support projections is defined by
\[
\mathcal{F}(\cM,\tau)=\{x\in \cM: \tau(s(x))<\infty\}.
\]

For any $x \in S(\mathcal{M}, \tau)$,  the {\it spectral 
distribution function} of $|x|$ is defined by setting
\[
d_{|x|}(\lambda) = \tau\bigl(e^{|x|}(\lambda, \infty)\bigr), \quad \lambda > 0.
\]
Note that $d_{|x|}$ is a right-continuous function (see, e.g., \cite{DPS}).
The {\it singular value function} is defined to be the right-continuous inverse of the spectral distribution function $d_{|x|}$, that is,
\[
\mu(t; x) = \inf\left\{ \lambda \geq 0 : d_{|x|}(\lambda) \leq t\right\}.
\]	
The space $S_0(\cM,\tau)$ of $\tau$-compact operators is
defined by
\[
S_0(\cM,\tau):=\{x\in S(\cM,\tau):\mu(\infty;x)=0\},
\]
which is an absolutely solid $*$-subalgebra of $S(\cM,\tau)$ \cite[Proposition 2.4.4]{DPS}.  
 Recall that
\begin{align*}
L_1(\cM,\tau):=\{x\in S(\cM,\tau):\tau(|x|)<\infty\}
=&\left\{x\in S(\cM,\tau):\int_0^\infty\mu(t;x)dt<\infty\right\}\\
=&\{x\in S(\cM,\tau):\mu(x)\in L_1(0,\infty )\}.
\end{align*}

For convenience of the reader, we also recall the definition of the measure topology $t_m$ on the algebra $S(\cM,\tau)$. For every $\varepsilon,\delta>0$, we define the neighborhood
\[
V(\varepsilon,\delta)=\{x\in S(\cM,\tau):\exists  p\in P(\cM) 
\text{ such that }  \norm{x(\mathbf{1}-p)}_\infty\leq \varepsilon, \tau(p)\leq \delta\}.
\]
The collection $\{V(\varepsilon,\delta):\varepsilon,\delta>0\}$ is a neighborhood base at zero  for a complete metrizable Hausdorff vector space topology $t_m$ on $S(\cM,\tau)$  (see, e.g., \cite{DPS}).
If a net $\{x_i\}_{i\in I}$ in $S(\cM,\tau)$ converges to the operator $x\in S(\cM,\tau)$ in $t_m$,
then this is denoted by $x_i\xrightarrow{t_m}x$ and the net $\{x_i\}_{i\in I}$ is said to converge to $x$ in measure.


The measure topology can also be characterized in terms of the singular value function.
For a net $\{x_i\}_{i\in I} \subset S(\cM,\tau)$, 
$x_i \xrightarrow{t_m}  0$ if and only if  
\begin{align}\label{m-mu}
    \mu(t;x_i )\to  0,\quad   t>0.
\end{align}

\subsection{$F$-normed Orlicz space}\label{Orlicz pre}

Recall that an extended-valued functional
$\rho:X\rightarrow[0,\infty]$ on a real linear space $X$ is called
a {\it modular}, 
if
\begin{enumerate}
\item $\rho(x)=0 \,\Longleftrightarrow\, x=0,$
\item $\rho(-x)=\rho(x),$
\item $\rho(\alpha x+\beta y)\leq \rho(x)+\rho(y)$
\end{enumerate}
for every $x,y\in X$ and every $\alpha,\beta\geq0$ such that
$\alpha+\beta=1$.

An $F$-norm $\norm{\cdot}$ on a complex linear space $X$ is a function $\norm{\cdot}:X\to [0,\infty)$, such that for all $x,y\in X$,
the following properties hold:
\begin{enumerate}
\item $\norm{x}=0\Longleftrightarrow x=0$;
\item $\norm{\alpha x}\le \norm{x}$, $\alpha\in\mathbb{C}$, $|\alpha|\le 1$;
\item $\lim_{\alpha\to 0}\norm{\alpha x}=0$;
\item $\norm{x+y}\le \norm{x}+\norm{y}$.
\end{enumerate}

Assume that $\Phi:[0,\infty)\to[0,\infty)$ is (not necessarily convex) continuous and non-decreasing   with $\Phi(0)=0$ and $\Phi(t)>0$ whenever $t>0$.
Let $m$ denote the Lebesgue measure on $(0,\infty)$, and let
$L_{0}(m)$ be the linear space of all complex-valued measurable functions on
$(0,\infty)$, where functions equal $m$-almost everywhere are
identified. 
Then,
\[
    \rho_{\Phi}(f):=
    \int_{0}^{\infty}\Phi\bigl(|f(t)|\bigr)dm(t),
    \quad f\in L_{0}(m).
\]
is a modular on $L_0(m)$ (see \cite[Section 2.3]{MO}). 
The linear subspace
\[
L_\Phi(0,\infty):=
\left\{f\in L_{0}(m):\rho_{\Phi}(\lambda f)<\infty\text{ for some } \lambda>0\right\}
\]
of $L_{0}(m)$ is called an {\it $F$-normed  Orlicz space},
which is complete with respect to the following $F$-norm,
called the {\it Mazur--Orlicz $F$-norm} (see \cite{MazurOrlicz, MO}):
\[
\cNorm{f}_\Phi:=\inf\left\{\lambda>0:\rho_\Phi\left(\frac{f}{\lambda}\right)\le \lambda\right\},\quad f\in L_\Phi(0,\infty).
\]


Let $\cM$ be a semifinite von Neumann algebra equipped with a semifinite faithful normal trace $\tau$.
By Calkin correspondence \cite{HLS},
\begin{align*}
L_\Phi(\cM,\tau):=&\{x\in S(\cM,\tau):\mu(x)\in L_\Phi(0,\infty)\}\\
=&\{x\in S(\cM,\tau):\tau(\Phi(\lambda|x|))<\infty\text{ for some }\lambda>0\}
\end{align*}
is an $F$-normed noncommutative  Orlicz space equipped with the Mazur--Orlicz $F$-norm 
\begin{align*}
\cNorm{x}_\Phi:=\cNorm{\mu(x)}_\Phi=&
\inf\left\{\lambda>0:\rho_\Phi\left(\frac{\mu(x)}{\lambda}\right)\le \lambda\right\}\\=&
\inf\left\{\lambda>0:\tau\left(\Phi\left(\left|\frac{x}{\lambda}\right|\right)\right)\le \lambda\right\},
\quad x\in L_\Phi(\cM,\tau).
\end{align*}
In this sense, the modular on $L_\Phi(\cM,\tau)$ is $\rho_\Phi(x)=\tau(\Phi(|x|))$, $x\in L_\Phi(\cM,\tau)$.
By the definition, we have $L_\Phi(\cM,\tau)\subset S_0(\cM,\tau)$.
Indeed, 
if $x\in L_\Phi(\cM,\tau)$, the fact $\mu(x)\in L_\Phi(0,\infty)$ implies that
there exists $\lambda>0$ such that
\[
\int_0^\infty \Phi(\lambda\mu(t;x))dm(t)<\infty.
\]
By the properties of $\Phi$, we have $\mu(\infty;x)=0$, i.e., $x\in S_0(\cM,\tau)$.

For each $F$-normed (noncommutative)  Orlicz space,
the following lemma presents the relationship between  the modular $\rho_\Phi$ and $F$-norm $\cNorm{\cdot}_\Phi$.

\begin{lemma}\label{modular-Fnorm con}
Let $\Phi:[0,\infty)\to [0,\infty)$ be continuous and non-decreasing with $\Phi(0)=0$, $\Phi(t)>0$ whenever $t>0$.
If $\lambda>0$ and $x\in L_\Phi(\cM,\tau)$, then
\[
\tau(\Phi(|x|))\le \lambda\, \Longleftrightarrow\,
\cNorm{\lambda x}_\Phi\le \lambda.
\]
\end{lemma}

\begin{proof}
The implication $\Rightarrow$ is trivial.

$(\Leftarrow)$.
Suppose that $\|\lambda x\|_{\Phi}\leq\lambda.$
Let $\{r_n\}$ be such that $\cNorm{\lambda x}_\Phi\le r_n<\lambda+\frac1n$. 
Then,
\begin{align}\label{mo-condition}
\tau\left(\Phi\left(\frac{\lambda}{r_n}|x|\right)\right)\leq r_n.
\end{align}
Denote
\[
a_n:=\frac{\lambda}{\lambda+\frac{1}{n}}<\frac{\lambda}{r_n}.
\]
Since $\Phi$ is non-decreasing, it follows that
\[
\Phi(a_n|x|)
\leq
\Phi\left(\frac{\lambda}{r_n}|x|\right).
\]
Consequently,
\begin{align}\label{anx mo-lim}
\tau\bigl(\Phi(a_n|x|)\bigr)\leq
\tau\left(\Phi\left(\frac{\lambda}{r_n}|x|\right)\right)
      \stackrel{\eqref{mo-condition}}{\le}r_n
    <\lambda+\frac1n.
\end{align}
Since $\Phi$ is continuous and non-decreasing, it follows that 
$\Phi(a_n|x|)\uparrow\Phi(|x|)$
whenever $n\to \infty$. By the normality of $\tau$, we have
\[
\tau(\Phi(|x|))=\lim_{n\to\infty}\tau(\Phi(a_n|x|))
\stackrel{\eqref{anx mo-lim}}{\le}
\lim_{n\to\infty}\left(\lambda+\frac1n\right)=\lambda,
\]
which completes the proof.
\end{proof}

\section{Jensen's trace inequality and its equality condition}

\subsection{The proof of Theorem \ref{equality condition} when $a$ is a $\tau$-finite projection}

The main result of this subsection is 
 Proposition \ref{prop:finite} below, 
 in which we consider  the equality condition in Jensen's inequality when the contraction $a$ is a $\tau$-finite projection.
This is the crucial ingredient in  the proof of Theorem \ref{equality condition}.



\begin{proposition}\label{prop:finite}
Assume that $\cM$ is a  semifinite von Neumann algebra equipped with a semifinite faithful normal trace $\tau$.
Let $e\in P(\cM)$ with $\tau(e)<\infty$ and let $x\in S(\cM,\tau)_h$. 
Suppose that $I$ is a non-degenerate interval containing $0$ and $\sigma(x),\sigma(exe)$. 
If $f:I\to\mathbb R$ is continuous and convex such that
$f(0)=0$ and $ef(x)e,f(exe)\in L_1(e\cM e,\tau)$,
then
\[
\tau(ef(x)e)\ge \tau(f(exe)).
\]
In the case where $f$ is strictly convex, equality holds if and only if $ex=xe$.
\end{proposition}

Scalarization through spectral measures underlies the proofs of Jensen's trace inequality in \cite{BrownKosaki, HaradaKosaki2010, Kosaki2013}. Precisely, for a contraction \(a\) and a unit vector \(\xi\in \cH\), one considers the probability measure
\footnote{$\delta_0$ denotes the Dirac probability measure concentrated at $0$.}
\[
\nu_\xi(B)
 =\|e^x(B)a\xi\|^2
  +\bigl(1-\|a\xi\|^2\bigr)\delta_0(B),
  \quad B\subset \mathbb{R},
\]
and, for suitable \(h\) with \(h(0)=0\),
\[
\langle a^*h(x)a\xi,\xi\rangle
   =\int h(t)\,d\nu_\xi(t).
\]
The desired vector inequality then follows from the classical scalar Jensen inequality, with approximation arguments used for unbounded operators under various  additional assumptions. Harada and Kosaki \cite[Theorem 15]{HaradaKosaki} applied the discrete version of this approaches to the equality problem for matrices by choosing eigenvectors of \(exe\).



Inspired by this, to prove Proposition \ref{prop:finite},
we aim to reduce  Jensen's trace inequality to a classical Jensen's inequality, thereby applying the equality condition.
To this end,
in Lemma \ref{lem:kernel}, we establish a new scalar integral
representation for the bounded Borel functional calculus of $exe$,
in the sense of trace preservation. 
For the convenience of the
reader, some   basic definitions are recalled.

Assume that $\cM$ is a von Neumann algebra equipped with a semifinite faithful normal trace $\tau$ and $\cN$ is a von Neumann subalgebra of $\cM$ such that the restriction $\tau|_\cN$ of $\tau$ to $\cN$ is semifinite.
There exists a unique linear {\it conditional expectation}
\[
E:(L_1+L_\infty)(\cM,\tau)\to (L_1+L_\infty)(\cN,\tau|_\cN),
\]
which is positive, faithful and trace-preserving.
Moreover,
\begin{align}\label{bimodular}
E(xy)=E(x)y,\quad E(ab)=aE(b)
\end{align}
for all $x\in L_1(\cM,\tau),\,y\in \cN,\, a\in L_1(\cN,\tau|_\cN),\, b\in \cM$.
(see, e.g., \cite{Takesaki,DPS,Umegaki1956}). 

For each $x\in S(\cM,\tau)_h$, 
denote by $W^*(x)$
the abelian von Neumann algebra generated by all spectral projections of $x$.
Assume, in addition, that $\tau|_{W^*(x)}$
is semifinite.
Then there exist a 
(strictly) localizable measure space
$(\Omega,\Sigma,\mu)$ (or simply, $(\Omega,\mu)$)
and normal a $*$-isomorphism
\[
\Psi:\left(W^*(x),\tau|_{W^*(x)}\right)\to L_\infty(\Omega,\mu)
\]
preserving trace (in the sense of integration), i.e.,
\[
\tau(y)=\int_\Omega \Psi(y)(\omega)\,d\mu(\omega),
\quad y\in W^*(x)_+
\]
(see e.g., \cite{KadisonRingroseI} and \cite[Theorem 7.21]{BGL2022}).

Let $(\Omega,\Sigma)$ and $(S,\mathcal S)$ be measurable spaces,
and let $\mathcal P(S)$ denote the space of probability measures
on $(S,\mathcal S)$.
A {\it probability kernel}  from
$\Omega$ to $S$ is a measurable mapping
\[
\omega\longmapsto\nu_\omega
\]
from $\Omega$ into $\mathcal P(S)$ (see
\cite[Chapter~3]{Kallenberg2021}). 
In the following lemma, we apply this notion with
$S=\sigma(x)$ equipped with its Borel $\sigma$-algebra
$\mathcal B(\sigma(x))$.

\begin{lemma}\label{lem:kernel}
Assume that $\cM$ is a  semifinite von Neumann algebra equipped with a semifinite faithful normal trace $\tau$.
Let $e\in P(\cM)$ with
$\tau(e)<\infty$, and let $x\in S(\cM,\tau)_h$.  
Denote
\[
        \cA:=W^*(exe)\subset e\cM e.
\]
Let $\Psi:\left(\cA,\tau|_\cA\right)\to L_\infty(\Omega,\mu)$ be a normal trace-preserving $*$-isomorphism, where $(\Omega,\mu)$ is a finite localizable measure space.
If $E_{\cA}:L_1(e\cM e,\tau)\to L_1(\cA,\tau|_\cA)$ is  the   conditional expectation,
then there exists a family $\{\nu_\omega\}_{\omega\in \Omega}$ of probability measures on $\sigma(x)$, forming a probability kernel,
such that
\[
\Psi(E_\cA(eh(x)e))(\omega)=\int_{\sigma(x)} h(t)d\nu_\omega(t),
\quad\text{for a.e. }\omega
\]
for every bounded Borel function $h$ on $\sigma(x)$.
\end{lemma}

\begin{proof}
Since $\Psi$ is normal, it follows from \cite[p. 457]{DPS} that
$\Psi\left(E_\cA(\cdot)\right)$ is normal on $e\cM e$.
For each $r\in\mathbb Q$, let
\[
f_r:=\Psi\left(E_{\mathcal A}
 \left(ee^x(-\infty,r)e\right)\right)\in L_\infty(\Omega,\mu).
\]
After discarding one null set $N\subset\Omega$,
we have
\[
0\leq f_r\leq\mathbf{1},\quad
f_r\leq f_s\ \ (r<s),\quad
f_r=\sup_{\substack{\mathbb{Q}\ni s<r}}f_s.
\]
Moreover, $f_r\to0$ pointwise as $r\to-\infty$ and $f_r\to\mathbf{1}$  pointwise as $r\to\infty$.

For every $\omega\in \Omega\backslash N$,
define
\[
F_\omega(t):=\sup_{\substack{\mathbb{Q}\ni r<t}}f_r(\omega),
\]
which is non-decreasing and left-continuous 
with $\lim\limits_{t\to -\infty}F_\omega(t)=0$ and 
$\lim\limits_{t\to +\infty}F_\omega(t)=1$. 
By the Lebesgue--Stieltjes theorem
\cite[Theorem 1.8.1]{Bogachev2007}, there exists a probability measure
$\nu_\omega$ on $\mathbb R$ such that
\[
\nu_\omega((-\infty,t))
 =F_\omega(t),
 \quad t\in\mathbb R.
\]
Note that $\omega\mapsto\nu_\omega(-\infty,r)$ on $\Omega\backslash N$ is measurable
when $r\in \mathbb{Q}$. 
Then, 
by the monotone-class theorem \cite[Theorem 1.9.3(ii)]{Bogachev2007}, 
we have $\omega\mapsto\nu_\omega(B)$ on $\Omega\backslash N$ is measurable for every Borel set $B$.
Since
\[
\Psi\left(E_{\mathcal A}
 \left(e\,e^x(-\infty,r)(x)e\right)\right)(\omega)
 =\nu_\omega((-\infty,r)),
\quad \omega\in \Omega\backslash N
\]
when $r\in\mathbb Q$,
it follows from the monotone-class theorem
\cite[Theorem 1.9.3(ii)]{Bogachev2007} that
\begin{align}\label{Phi=nu Borel}
\Psi\left(E_{\mathcal A}\left(ee^x(B)e\right)\right)(\omega)
 =\nu_\omega(B),
\quad\omega\in \Omega\backslash N
\end{align}
for every Borel set $B\subset \mathbb{R}$.
In particular,
\[
0=\Psi\left(E_{\mathcal A}\left(ee^x(\mathbb{R}\backslash \sigma(x))e\right)\right)(\omega)
 =\nu_\omega(\mathbb{R}\backslash \sigma(x)),
\quad\omega\in \Omega\backslash N,
\]
which implies that $\nu_\omega$ is a probability measure on $\sigma(x)$ when $\omega\in \Omega\backslash N$.
Let $h$ be a bounded simple function on $\sigma(x)$.
It follows from \eqref{Phi=nu Borel} that
\[
\Psi\left(E_{\mathcal A}(eh(x)e)\right)(\omega)
=\int_{\sigma(x)}h(t)d\nu_\omega(t),
\quad\omega\in \Omega\backslash N.
\]
Let $h$ be a bounded Borel function on $\sigma(x)$.
Noting that
\[
\left\|
\Psi\left(E_{\mathcal A}\left(eh(x)e\right)\right)
\right\|_\infty
\leq \|h\|_\infty,
\]
by uniform
approximation by simple functions, we have
\[
\Psi\left(E_{\mathcal A}(eh(x)e)\right)(\omega)
=\int_{\sigma(x)}h(t)d\nu_\omega(t),
\quad\omega\in \Omega\backslash N.
\]

Let $\rho$ be an arbitrary fixed probability measure on $\sigma(x)$ and let $\nu_\omega=\rho$ for all $\omega\in N$.
Recall that $\omega\mapsto\nu_\omega(B)$ on $\Omega\backslash N$ is measurable for every Borel set $B$.
It is readily verified that 
\[
\omega\longmapsto\nu_\omega
\]
is measurable mapping from $\Omega$ to $\cP(\sigma(x))$.
Thus, $\{\nu_\omega\}_{\omega\in\Omega}$
is the required probability kernel.
\end{proof}

The following lemma is trivial when $y\in L_1(\cA,\tau|_\cA)_h$. Below, we relax the condition to $y\in S(\cA,\tau|_\cA)_h$. 
\begin{lemma}\label{lem:localized-positivity}
Assume that $\cM$ is a von Neumann algebra equipped with a finite faithful normal trace $\tau$. 
Let
\(\mathcal A\) be a von Neumann subalgebra of $\cM$,
and let $E_\cA:L_1(\cM,\tau)\to L_1(\cA,\tau|_\cA)$ be the conditional expectation. 
If $x\in L_1(\cM,\tau)_h$, $y\in S(\cA,\tau|_\cA)_h$
satisfy \(x+y\geq0\), 
then
\[
    E_{\mathcal A}(x)+y\geq0.
\]
\end{lemma}

\begin{proof}

For every \(n\geq1\), let $q_n:=e^{|y|}[0,n)\in\mathcal A$.
Since \(\tau\) is finite,
it follows that $q_nyq_n\in \cA\subset  L_1(\mathcal A,\tau|_{\mathcal A})$,
which implies that
\[
    q_n(x+y)q_n\in L_1(\mathcal M,\tau)_+.
\] 
By \eqref{bimodular} and $q_n\in \cA$, we obtain 
\[
0\leq E_{\mathcal A}(q_n(x+y)q_n)
=q_nE_{\mathcal A}(x)q_n+q_nyq_n
=q_n(E_{\mathcal A}(x)+y)q_n.
\]
Noting that \(q_n\uparrow\mathbf{1}\), 
by \cite[Proposition 2.6.11(i)]{DPS}
we have
\[
    q_n(E_{\mathcal A}(x)+y)q_n
    \xrightarrow{t_m} E_{\mathcal A}(x)+y,
\]
which together with \cite[Proposition 2.6.1(i)]{DPS} yields that
$E_{\mathcal A}(x)+y\geq0.$
\end{proof}

Having the above two lemmas at hand, we are able to prove Proposition \ref{prop:finite}.
\begin{proof}[Proof of Proposition \ref{prop:finite}]

Let $\cA:=W^*(exe)\subset e\cM e$
and let 
$$E_\cA:L_1(e\cM e,\tau|_{e\cM e})\to L_1(\cA,\tau|_\cA)$$ be the conditional expectation.
Denote\footnote{${\rm int}I$ denotes the set of interior points of $I$.}
\[
D:=\mathbb{Q}\cap {\rm int}I
=\{t_k\}_{k\ge 1}.
\]
Recall that $f$ is convex and continuous.
Let $\{l_k\}_{k\ge 1}$ be a sequence of affine minorants such that
\[
l_k\le f,\quad l_k(t_k)=f(t_k),\quad k\ge 1,
\]
which implies that 
\begin{align}\label{f sup lk}
f(t)=\sup_{k\ge 1}l_k(t),\quad t\in I.
\end{align}
For every $k,n\ge 1$, set 
\[
g_k(t):=f(t)-l_k(t)\ge0,\quad g_{k,n}(t):=g_k(t)\wedge n.
\]
For each $k\ge 1$,
since $l_k$ is an affine minorant, it follows from the functional calculus that
$el_k(x)e=l_k(exe)$, 
which together with \cite[Proposition 2.2.20]{DPS} yields that
\[
(ef(x)e-eg_{k,n}(x)e)-l_k(exe)=e(f-g_{k,n}-l_k)(x)e
=e(g_k-g_{k,n})(x)e\ge 0,\quad n\ge 1.
\]
Noting that
$ef(x)e-eg_{k,n}(x)e\in L_1(e\cM e,\tau|_{e\cM e})$,
by Lemma \ref{lem:localized-positivity}, we obtain 
\begin{align}\label{f-lk ge gkn}
E_{\mathcal A}(ef(x)e)-l_k(exe)
\geq E_{\mathcal A}(eg_{k,n}(x)e)
\geq 0.
\end{align}

Let $\Psi:(\cA,\tau|_\cA)\to L_\infty(\Omega,\mu)$ be a trace-preserving $*$-isomorphism, where $(\Omega,\mu)$ is a finite localizable measure space.
Denote the linear space of all complex-valued measurable functions on $(\Omega,\mu)$ by $L_0(\Omega,\mu)$.
By \cite[Theorem 3.13]{weigt},   $\Psi$ is continuous with respect to measure topology.
It follows from \cite[Proposition 2.9.2]{DPS} that
$\Psi$ can uniquely extend to a $*$-isomorphism
\[
\Psi:S(\cA,\tau|_\cA)\to L_0(\Omega,\mu).
\]
In particular, $\Psi$ is trace-preserving \cite[Proposition 3.3.10]{DPS}. 
By \cite[Proposition~2.9.2(iii)]{DPS} and Lemma \ref{lem:kernel}, there exists a family $\{\nu_\omega\}_{\omega\in \Omega}$ of probability measures on $\sigma(x)$, forming a probability kernel, such that
for all $n\ge 1$ and $k\ge 1$,
\begin{align*}
\Psi(E_\cA(ef(x)e))(\omega)-l_k(\Psi(exe)(\omega))
=&
\Psi(E_\cA(ef(x)e))(\omega)-\Psi(l_k(exe))(\omega)\\
\stackrel{\eqref{f-lk ge gkn}}{\ge}&
\Psi(E_\cA(eg_{k,n}(x)e))(\omega)\\
=&\int_{\sigma(x)} g_{k,n}(t)d\nu_\omega(t)\ge 0
\quad\text{for a.e. }\omega.
\end{align*}
Applying the monotone convergence theorem 
\cite[Theorem 1.26]{Rudin},
it follows that
\begin{equation}\label{eq:unbounded mi}
\begin{aligned}
\Psi(E_\cA(ef(x)e))(\omega)-l_k(\Psi(exe)(\omega))
\ge& \int_{\sigma(x)} g_k(t)d\nu_\omega(t)\\
=&\int_{\sigma(x)} f(t)-l_k(t)d\nu_\omega(t)\ge 0
\end{aligned}
\end{equation}
for almost every $\omega$.
By \eqref{f sup lk} and \cite[Proposition 2.9.2(iii)]{DPS}, we have 
\begin{align}\label{EAefxe ge fexe}
\Psi(E_\cA(ef(x)e))(\omega)\ge f(\Psi(exe)(\omega))
=\Psi(f(exe))(\omega)
\quad\text{for a.e. }\omega.
\end{align}
Recall that $\Psi$ and $E_\cA$ are trace-preserving.
Then,
\[
\tau(ef(x)e)=
\int_\Omega\Psi(E_\cA(ef(x)e))(\omega)d\mu(\omega)
\ge \int_\Omega\Psi(f(exe))(\omega)d\mu(\omega)
=\tau(f(exe)).
\]

Next, we discuss the cases where equality holds.
Suppose that $f$ is strictly convex and equality in the above formula holds. 
This together with \eqref{EAefxe ge fexe} implies that 
$$E_\cA(ef(x)e)=f(exe). $$
By \cite[Proposition 2.9.2(iii)]{DPS}, we have 
\[
\Psi(E_\cA(ef(x)e))(\omega)=\Psi(f(exe))(\omega)=f(\Psi(exe)(\omega))\quad\text{for a.e. }\omega.
\]
For each $k\ge 1$, by \eqref{eq:unbounded mi}, we have
\begin{align}\label{eq:int le}
\int_{\sigma(x)} f(t)-l_k(t)d\nu_\omega(t)\le
f(\Psi(exe)(\omega))-l_k(\Psi(exe)(\omega))
\quad\text{for a.e. }\omega.
\end{align}
Since $\{l_k\}_{k\ge 1}$ is countable, it follows that after discarding a null set $N\subset\Omega$, \eqref{eq:int le} holds for all $k\ge 1$.
Let $\omega\in \Omega\backslash N$ be arbitrary.
By \eqref{f sup lk}, there exists a sequence $\{k_j\}$ such that 
\begin{align}\label{limit point}
l_{k_j}(\Psi(exe)(\omega))\to f(\Psi(exe)(\omega)),
\end{align}
which together with \eqref{eq:int le} yields that
\[
\int_{\sigma(x)} f(t)-l_{k_j}(t)d\nu_\omega(t)\to 0.
\]
Since $f$ is strictly convex, $f\ge l_{k_j}$ and $l_{k,j}$ are affine minorants, it follows from 
Fatou's Lemma that
\[
0\le\int_{\sigma(x)}
\liminf_j\left(f(t)-l_{k_j}(t)\right)d\nu_\omega(t)
\le 0,
\]
which together with
\eqref{limit point} yields that 
\footnote{$\delta_{\Phi(exe)(\omega)}$ denotes the Dirac probability measure concentrated at $\Psi(exe)(\omega)$.}
\begin{align}\label{measure trans}
 \nu_\omega=\delta_{\Psi(exe)(\omega)}
\quad\text{for a.e. }\omega.
\end{align}
Let Borel set $B\subset\sigma(x)$ be arbitrary. 
Noting that
\begin{align*}
\Psi(E_\cA(e\chi_B(x)e))(\omega)
\stackrel{\text{Lemma \ref{lem:kernel}}}{=}
\int_{\sigma(x)}\chi_B(t)d\nu_\omega(t)
\stackrel{\eqref{measure trans}}{=}&
\int_{\sigma(x)}\chi_B(t)d\delta_{\Psi(exe)(\omega)}(t)\\
=&\chi_B(\Psi(exe)(\omega))
\quad\text{for a.e. }\omega,
\end{align*}
it follows that
\[
\Psi(E_\cA(e\chi_B(x)e))=\chi_B(\Psi(exe)) 
\stackrel{\text{\cite[Prop. 2.9.2(iv)]{DPS}}}{=}
\Psi(\chi_B(exe)),
\]
i.e., $E_\cA(e\chi_B(x)e)=\chi_B(exe)$.
We claim that $e\chi_B(x)e=\chi_B(exe)\in P(\cA)$.
Indeed,
denote
\[
a:=e\chi_B(x)e\le e,\quad p:=\chi_B(exe)\in P(\cA).
\]
Then, $E_\cA(a)=p$ and
\[
E_\cA((e-p)a(e-p))\stackrel{\eqref{bimodular}}{=}
(e-p)E_\cA(a)(e-p)
=(e-p)p(e-p)=0,
\]
which implies that $(e-p)a(e-p)=0$,
i.e., $a^\frac{1}{2}(e-p)=0$.
This implies $ap=a$.
Arguing similarly, we have $p(e-a)p=0$ and hence,
$p=ap$.
Consequently,
\[
 e\chi_B(x)e=a=ap=p=\chi_B(exe)\in P(\cA), 
\]
which proves our claim.
Noting that
\[
0=e\chi_B(x)e-(e\chi_B(x)e)^2=e\chi_B(x)(\mathbf{1}-e)\chi_B(x)e
=|(\mathbf{1}-e)\chi_B(x)e|^2,
\]
we have $(\mathbf{1}-e)\chi_B(x)e=0$ and
$e\chi_B(x)(\mathbf{1}-e)=0$. 
Hence, 
\[
e^x(B)=\chi_B(x)=e\chi_B(x)e+(\mathbf{1}-e)\chi_B(x)(\mathbf{1}-e)
\]
commutes with $e$.
Since $B$ is arbitrary,
it follows from \cite[Proposition 2.2.22]{DPS} that
$x$ commutes with $e$.

Conversely, if $x$ commutes with $e$, then $e\cH$ and $(\mathbf{1}-e)\cH$ are reducing space of $x$, i.e., 
$x=exe+(\mathbf{1}-e)x(\mathbf{1}-e)$.
It follows that
\begin{align}\label{reducing space}
f(exe)=ef(x)e.
\end{align}
Hence,
\[
 \tau(f(exe))=\tau(ef(x)e).
\]
This completes the proof.
\end{proof}

\subsection{Equality conditions for Jensen's inequality}

Throughout this subsection, unless stated otherwise, we assume that $\cM$ is a semifinite von Neumann algebra equipped with a semifinite faithful normal trace $\tau$.

The following lemma is an easy consequence of 
 \cite[Proposition 3.4.30]{DPS}.
\begin{lemma}\label{lem:blocks}
Let $e,q\in P(\cM)$ such that $q\le e$. 
If $x\in L_1(e\cM e,\tau)$, then
\[
        \tau(x)=\tau(qxq)+\tau((e-q)x(e-q)).
\]
\end{lemma}


To extend the result of Proposition \ref{prop:finite} to general projections, we need the following two lemmas, which provide useful approximation tools.

\begin{lemma}\label{lem:L1cut}
Let $e\in P(\cM)$ and
$\{q_\alpha\}\subset P(\cM)$ such that
$q_\alpha\uparrow e$.  
If $x\in L_1(\cM,\tau)$, then
\[
\tau(exe-q_\alpha xq_\alpha)\to 0.
\]
\end{lemma}


The following lemma is obvious for the special case when $x\in S_0(\cM,\tau)$. 
\begin{lemma}\label{lem:exhaust}

Let \(e\in P(\mathcal M)\), \(x\in S(e\mathcal Me,\tau)_h\), and let \(I\subset\mathbb R\) be an interval containing \(\sigma(x)\). Suppose that \(f:I\to\mathbb R\) is continuous and convex, and that \(f\) is either strictly convex or takes a negative value somewhere on $I$.
If 
$f(x)\in L_1(e\cM e,\tau)$,
then there exists a directed net $\{q_\alpha\}\subset P(e\cM e)$
such that
\[
q_\alpha\uparrow e,\quad
\tau(q_\alpha)<\infty,\quad
q_\alpha x=xq_\alpha.
\]
\end{lemma}

\begin{proof}
For each $1\le n\in \mathbb{N}$, denote $p_n:=e^{|f(x)|}\left(\frac{1}{n},\infty\right)$.
Since $f(x)\in L_1(e\cM e,\tau)$,
it follows that 
\[
\tau(p_n)<\infty, \quad p_n\uparrow s(f(x)).
\]
Let $z:=e-s(f(x))$. 
Then,
$z=e^x\{\lambda\in\mathbb{R}:f(\lambda)=0\}$.
By properties of $f$, we have that the zero set of $f$ contains at most two points. 
Consequently, 
there exist two spectral projections $z_1,z_2$ (may be zero) of $x$ satisfying
\begin{align}\label{z=z1+z2}
z=z_1+z_2,\quad xz=\lambda_1z_1+\lambda_2z_2,\quad\lambda_1,\lambda_2\in \mathbb{R}.
\end{align}
There exist two increasing nets 
$\{r_{1,\beta_1}\},\{r_{2,\beta_2}\}$ of $\tau$-finite projections satisfying
\[
r_{1,\beta_1}\le z_1,\quad r_{1,\beta_1}\uparrow z_1,\quad
r_{2,\beta_2}\le z_2,\quad r_{2,\beta_2}\uparrow z_2.
\]
Noting that all $r_{1,\beta_1}$ and $r_{2,\beta_2}$ commutes with $x$ (see \eqref{z=z1+z2}), we have 
\[
q_{n,\beta_1,\beta_2}:=p_n+r_{1,\beta_1}+r_{2,\beta_2}
\]
form the required directed net.

\end{proof}

The following proposition extends the result in Proposition~\ref{prop:finite} to the case of general projections.


\begin{proposition}\label{prop:semifinite}

Let $e\in P(\cM)$ and $x\in S(\cM,\tau)_h$. 
Suppose that $I$ is a non-degenerate interval containing $0$ and $\sigma(x),\sigma(exe)$. 
If $f:I\to\mathbb R$ is continuous and convex such that
$f(0)=0$ and $ef(x)e,f(exe)\in L_1(e\cM e,\tau)$,
then
\[
\tau(f(exe))\le \tau(ef(x)e).
\]
In the case where $f$ is strictly convex, equality holds if and only if $ex=xe$.
\end{proposition}

\begin{proof}



In the case where $f\ge 0$, the inequality follows from \cite[Theorem 3.5]{HaradaKosaki2010}.
It suffices to consider the case where $f$ is not non-negative.

Let $\{q_\alpha\}$ be the increasing net from Lemma \ref{lem:exhaust} such that $q_\alpha$ commutes with $exe$.
For every $q\in \{q_\alpha\}$, define
\[
\Delta(q):=\tau(qf(x)q)-\tau(qf(exe)q).
\]
Since $q$ commutes with $exe$, similarly to \eqref{reducing space}, it follows that
\begin{align}\label{fqxq=qfexeq}
f(qxq)=f(qexeq)=qf(exe)q,
\end{align}
which together with Proposition \ref{prop:finite} yields that
\begin{align}\label{eq:Deltaq >=0}
\Delta(q)=\tau(qf(x)q)-\tau(f(qxq))\ge0.
\end{align}
If $p,q\in \{q_\alpha\}$ with $p\le q$, then
Lemma \ref{lem:blocks} and Proposition \ref{prop:finite} imply that
\begin{align*}
\Delta(q)-\Delta(p)&=\tau((q-p)f(x)(q-p))-\tau((q-p)f(exe)(q-p))\\
&\stackrel{\eqref{fqxq=qfexeq}}{=}
\tau((q-p)f(x)(q-p))-\tau(f((q-p)x(q-p)))\ge 0.
\end{align*}
Thus, $\Delta(q_\alpha)\uparrow$.
Since $q_\alpha\uparrow e$, it follows from Lemma \ref{lem:L1cut} that
\[
\tau(ef(x)e-q_\alpha f(x)q_\alpha)\to 0
\,\text{ and }\,
\tau(f(exe)-q_\alpha f(exe)q_\alpha)\to 0,
\]
which imply that $\Delta(q_\alpha)\to \tau(ef(x)e)-\tau(f(exe))$.
By \eqref{eq:Deltaq >=0}, we have
\[
\tau(ef(x)e)\ge \tau(f(exe)).
\]

Next, we discuss the cases where equality holds.
Suppose that $f$ is strictly convex.
If
\[
\tau(ef(x)e)= \tau(f(exe)),
\]
then  constructing $\{q_\alpha\}$ again by Lemma \ref{lem:exhaust},  we obtain 
$$0\le\Delta(q_\alpha)=\tau(q_\alpha f(x)q_\alpha)-\tau(f(q_\alpha xq_\alpha))\uparrow \tau(ef(x)e)-\tau(f(exe))=0,$$
which implies that
\[
\tau(q_\alpha f(x)q_\alpha)=\tau(f(q_\alpha xq_\alpha))
\]
for all $\alpha$.
Applying Proposition~\ref{prop:finite} to $q_\alpha$,
it follows that $q_\alpha$ commutes with $x$.
Noting that $q_\alpha\xrightarrow{so}e$, we have $e$ commutes with $x$.

Conversely, if $ex=xe$, then by an argument analogous to the proof of Proposition~\ref{prop:finite}, we obtain
\[
\tau(ef(x)e)= \tau(f(exe)).
\]
\end{proof}

For $x$ and $a$ as in Theorem \ref{equality condition},
by the dilation technique used in the proof of 
\cite[Corollary 3]{HaradaKosaki},
we may apply Proposition \ref{prop:semifinite} to the dilated self-adjoint operator and the corresponding projection in
$S\left(\cM\bar\otimes M_2,\tau\otimes\Tr\right)$, 
and thereby prove Theorem~\ref{equality condition}.
For completenss, we include a full proof below.


\begin{proof}[Proof of Theorem \ref{equality condition}]

Let
\[
        \tM=\cM\bar\otimes M_2,
        \quad
        \ttau=\tau\otimes\Tr.
\]
Define
\[
        \tx:=
        \begin{pmatrix}
        x&0\\
        0&0
        \end{pmatrix}
        \in S(\tilde{\cM},\tilde{\tau}),
        \quad
        v=
        \begin{pmatrix}
        a\\
        (\mathbf{1}-a^*a)^{1/2}
        \end{pmatrix}.
\]
Noting that
\[
f(\tx)=\begin{pmatrix}
        f(x)&0\\
        0&0
        \end{pmatrix}
\]
($f(0)=0$), we have 
\begin{align}\label{x tx trans}
v^*\tx v=a^*xa,\quad v^*f(\tx)v= a^*f(x)a.
\end{align}
Let $e:=vv^*\in \tM$.
Then we have $e^*=e$ and 
\[
e^2=(vv^*)(vv^*)=v(v^*v)v^*=v\mathbf{1}v^*=vv^*=e,
\]
i.e., $e\in P(\tM)$.
Define $\Psi:\cM\to e\tM e$ by
\[
\Psi(y):=vyv^*,\quad y\in \cM,
\]
which is a unital $*$-isomorphism ($v^*v=\mathbf{1}$).
For each $y\in \cM$, since
\begin{equation}\label{Psi tracepreserving}
\begin{aligned}
\ttau(\Psi(y))=
\ttau(vyv^*)
&=\tau(aya^*)+\tau((\mathbf{1}-a^*a)^{1/2}y(\mathbf{1}-a^*a)^{1/2})\\
&=\tau(ya^*a)+\tau(y(\mathbf{1}-a^*a))=\tau(y),
\end{aligned}
\end{equation}
it follows from \cite[Proposition 2.9.3]{DPS} that 
$\Psi$ can uniquely extended to a trace-preserving $*$-isomorphism
\[
\Psi:S(\cM,\tau)\to S(e\tM e,\ttau).
\]
Noting that 
\begin{align}\label{tx x trans}
    e\tx e\stackrel{\eqref{x tx trans}}{=}v(a^*xa)v^*,
\end{align}
by \cite[Proposition 2.9.2]{DPS}, we have
\begin{align}\label{ftx fx trans}
    f(e\tx e)
\stackrel{\eqref{tx x trans}}{=}
f(\Psi(a^*xa))=\Psi(f(a^*xa))
=vf(a^*xa)v^*.
\end{align}
It follows from \eqref{Psi tracepreserving} that
\begin{align}\label{ttau tau}
\ttau(ef(\tx)e)
\stackrel{\eqref{x tx trans}}{=}
\tau(va^*f(x)av^*)=
\tau(a^*f(x)a),
\quad
\ttau(f(e\tx e))
\stackrel{\eqref{ftx fx trans}}{=}
\tau(f(a^*xa)),
\end{align}
which together with Proposition~\ref{prop:semifinite} yields that
\[
\tau(a^*f(x)a)\ge \tau(f(a^*xa)).
\]

Next, we consider the case where $f$ is strictly convex.
Assume first that $\tau(f(a^*xa))=\tau(a^*f(x)a)$.
Then, 
\[
\ttau(ef(\tx)e)\stackrel{\eqref{ttau tau}}{=}\ttau(f(e\tx e)).
\]
Similarly to \eqref{tx x trans}, we have $v(a^*x^2a)v^*=e\tx^2e$,
which together with Proposition \ref{prop:semifinite} yields that
\[
\Psi(a^*x^2a)=e\tx^2 e=(e\tx e)^2
\stackrel{\eqref{tx x trans}}{=}\Psi\left((a^*xa)^2\right),
\]
i.e., $a^*x^2a=(a^*xa)^2$.

Conversely, assume that $a^*x^2a=(a^*xa)^2$.
Then,
\[
e\tx^2 e=\Psi\left(a^*x^2a\right)=\Psi\left((a^*xa)^2\right)\stackrel{\eqref{tx x trans}}{=}(e\tx e)^2.
\]
Consequently,
\begin{align*}
\left|\left(\mathbf{1}_{\tM}-e\right)\tilde{x} e\right|^2=e\tilde{x}\left(\mathbf{1}_{\tM}-e\right)\tilde{x} e=0,
\end{align*}
i.e., $\left(\mathbf{1}_{\tM}-e\right)\tilde{x} e=0=e\tx\left(\mathbf{1}_{\tM}-e\right)$.
Note that $\tx$ commutes with $e$.
Arguing  mutatis mutandis as in the  proof of Proposition \ref{prop:finite}, we have 
\[
\tau(f(a^*xa))\stackrel{\eqref{ttau tau}}{=}
\ttau(f(e\tx e))=\ttau(ef(\tx)e)
\stackrel{\eqref{ttau tau}}{=}\tau(a^*f(x)a).
\]

\end{proof}

Let $y,z\in S(\cM,\tau)$. 
If
\[
\int_0^t \mu(s;y)ds\le \int_0^t \mu(s;z)ds,\quad t>0,
\]
then $y$ is said to be {\it submajorized} by $z$, denoted by
$y\prec\!\prec z$\cite{DPS,LSZ}.

\begin{remark}\label{positive case}
Let $a$, $x$ and $f$ be as in Theorem~\ref{equality condition}.
\begin{enumerate}
\item 
In the special case where 
$f$ is non-negative, 
the submajorization inequality
of \cite[Theorem 3.5]{HaradaKosaki2010} yields
\begin{align}\label{sub f}
f(a^*xa)\prec\!\prec a^*f(x)a,
\end{align}
i.e.,
\[
\int_0^t\mu(s;f(a^*xa))\,ds
\leq
\int_0^t\mu(s;a^*f(x)a)\,ds,
\quad t>0.
\]
If, in addition, $f$ is strictly convex, then 
\begin{align}\label{trace singular value equality}
\tau(f(a^*xa))
 =\tau(a^*f(x)a)
\Longleftrightarrow 
\mu(f(a^*xa))
=
\mu(a^*f(x)a).
\end{align}
Indeed, if equality holds in Theorem~\ref{equality condition}, then \eqref{ttau tau} and
Proposition~\ref{prop:semifinite} imply that $e$ commutes with $\widetilde{x}$ (in the proof of Theorem \ref{equality condition}), which together with Borel functional calculus yields that $f(e\widetilde{x}e)=ef(\widetilde{x})e$.
Hence,
by \eqref{x tx trans} and \eqref{ftx fx trans}, we have
\[
\Psi(f(a^*xa))
 =f(e\widetilde{x}e)
 =ef(\widetilde{x})e
 =\Psi(a^*f(x)a).
\]
Since $\Psi$ is injective, it follows that
\[
f(a^*xa)=a^*f(x)a,
\]
and therefore their singular value functions coincide. Conversely,
the equality of the singular value functions implies
\[
\tau(f(a^*xa))
 =\int_0^\infty\mu(t;f(a^*xa))\,dt
 =\int_0^\infty\mu(t;a^*f(x)a)\,dt
 =\tau(a^*f(x)a).
\]

\item 
The non-negativity assumptions on $f $ in the first part of this remark  are essential.
Indeed, \cite[Theorem 3.5]{HaradaKosaki2010}  yields only 
\[
f(a^*xa)_+\prec\!\prec \bigl(a^*f(x)a\bigr)_+.
\]
As pointed out in
\cite[p.~141]
{HaradaKosaki2010}, \eqref{sub f}
cannot be expected   when  $f$ is  not non-negative.
On the other hand, the non-negativity assumptions are also essential
for the equivalence \eqref{trace singular value equality}.

\end{enumerate}
\end{remark}

\subsection{Related inequalities and their equality conditions.}

Following the dilation technique used in the proof of \cite[Corollary 8]{HaradaKosaki},
we apply
Theorem \ref{equality condition} to characterize the equality cases in the following two trace inequalities.

\begin{corollary}\label{prop:two contractions}
Let $\{x_1,\dots,x_n\}\subset S(\cM,\tau)_h$ and let $\{a_1,\dots,a_n\}\subset\cM$ be such that 
$\sum_{k=1}^na_k^*a_k\le \mathbf{1}$.
Suppose that $I$ is a non-degenerate interval containing $0$ and 
\(\bigcup_{k=1}^n\sigma(x_k)\),
\(\sigma\left(\sum_{k=1}^n a_k^*x_ka_k\right)\).
Assume that $f$ is continuous on $I$ with $f(0)=0$
and 
\[
f\left(\sum_{k=1}^n a_k^*x_ka_k\right),
\quad
a_k^*f(x_k)a_k\,~(1\le k\le n)
\]
belong to $L_1(\cM,\tau)$.
If $f$ is convex, then
\[
\tau\left(f\left(\sum_{k=1}^n a_k^*x_ka_k\right)\right)\le 
\sum_{k=1}^n \tau(a_k^*f(x_k)a_k).
\]
In the case when $f$ is strictly convex, equality holds if and only if 
\[
\sum_{k=1}^na_k^*x_k^2a_k=\left(\sum_{k=1}^na_k^*x_ka_k\right)^2.
\]

If $f$ is concave, then the reverse inequality holds. 
In the case when $f$ is strictly concave,
equality again holds if and only if
\[
\sum_{k=1}^na_k^*x_k^2a_k=\left(\sum_{k=1}^na_k^*x_ka_k\right)^2.
\]
\end{corollary}

\begin{proof}

Consider the algebra
$\cM\bar\otimes M_n$ with trace $\tau\otimes\Tr$.
Denote
\[
        z:={\rm diag}(x_1,\dots,x_n),
        \quad
        v:=
        \begin{pmatrix}
        a_1 & 0 & \cdots & 0\\
        a_2 & 0 & \cdots & 0\\
        \vdots & \vdots & \cdots & \vdots\\
        a_n & 0 & \cdots & 0\\
        \end{pmatrix}.
\]
Note that $v^*v={\rm diag}\left(\sum_{k=1}^na_k^*a_k, 0,\dots, 0\right)
\le \mathbf{1},$
which implies that $v$ is a contraction in $\cM\bar\otimes M_n$.
Since $f(0)=0$, it follows that
\[
        f(v^*zv)={\rm diag}\left(f\left(\sum_{k=1}^na_k^*x_ka_k\right), 0,\dots, 0\right)
\]
and 
\[
        v^*f(z)v={\rm diag}\left(\sum_{k=1}^na_k^*f(x_k)a_k, 0,\dots, 0\right).
\]

If $f$ is convex, then \cite[Theorem 3.4]{HaradaKosaki2010} (or Theorem \ref{equality condition}) implies that
\begin{align*}
\tau\left(f\left(\sum_{k=1}^na_k^*x_ka_k\right)\right)=
(\tau\otimes\Tr)(f(v^*zv))\le&
(\tau\otimes\Tr)(v^*f(z)v)\\
=&\sum_{k=1}^n \tau(a_k^*f(x_k)a_k).
\end{align*}
Assume that $f$ is strictly convex.
By Theorem \ref{equality condition}, equality holds
if and only if 
$v^*z^2v=(v^*zv)^2$.
Hence,
\[
\tau\left(f\left(\sum_{k=1}^na_k^*x_ka_k\right)\right)= 
\sum_{k=1}^n \tau(a_k^*f(x_k)a_k)
\]
if and only if   $v^*z^2v=(v^*zv)^2$ if and only if 
\[
\sum_{k=1}^na_k^*x_k^2a_k=\left(\sum_{k=1}^na_k^*x_ka_k\right)^2.
\]

The concave case follows by the same argument applied to $-f$.
\end{proof}

\begin{remark}
Let $f$, $\{x_1,\dots,x_n\}$ and $\{a_1,\dots,a_n\}$ be as in Corollary~\ref{prop:two contractions}.
In the convex case, the block-matrix construction used in the proof
of Corollary~\ref{prop:two contractions}, together with \cite[Corollary 3.6]{HaradaKosaki2010},
yields the following submajorization counterpart of its trace
inequality:
\[
\left(
f\left(\sum_{k=1}^n a_k^*x_ka_k\right)
\right)_+
\prec\!\prec
\left(
\sum_{k=1}^n a_k^*f(x_k)a_k
\right)_+.
\]
If, in addition, $x_k\geq0$ for every $k$ and $f$ is non-negative on
$I\cap[0,\infty)$, then
\[
f\left(\sum_{k=1}^n a_k^*x_ka_k\right)
\prec\!\prec
\sum_{k=1}^n a_k^*f(x_k)a_k.
\]
When $f$ is, in addition, strictly convex and the operators above belong to $L_1(\cM,\tau)$,
equality in Corollary \ref{prop:two contractions}
holds if and
only if
\[
\mu\left(t;f\left(\sum_{k=1}^n a_k^*x_ka_k\right)\right)
=\mu\left(t;\sum_{k=1}^n a_k^*f(x_k)a_k\right),\quad t>0.
\]
Indeed, suppose first that equality holds in Corollary~\ref{prop:two contractions}.
With $z$ and $v$ as in its proof, we have
\[
 v^*z^2v=(v^*zv)^2.
\]
It follows from Remark \ref{positive case} that
$ f(v^*zv)=v^*f(z)v$, i.e.,
\[
 f\left(\sum_{k=1}^{n}a_k^*x_ka_k\right)
 =
 \sum_{k=1}^{n}a_k^*f(x_k)a_k.
\]
Consequently,
\[
 \mu\left(t;
 f\left(\sum_{k=1}^{n}a_k^*x_ka_k\right)\right)
 =
 \mu\left(t;
 \sum_{k=1}^{n}a_k^*f(x_k)a_k\right),
 \qquad t>0.
\]
The converse implication is trivial.

In the concave case, the corresponding trace inequality in
Corollary~\ref{prop:two contractions} also admits a reverse submajorization refinement under
an additional monotonicity assumption.
More precisely,
if, in addition, $x_k\geq0$ for every $k$ and $f$ is non-negative and non-decreasing on
$I\cap[0,\infty)$, then 
applying
\cite[Lemma 3.2(ii)]{HaradaKosaki2010} (or \cite[Lemma 10]{BrownKosaki}) to $-f$ and using the same
block-matrix construction, we have
\[
 \mu\left(
 t;\sum_{k=1}^{n}a_k^*f(x_k)a_k
 \right)
 \leq
 \mu\left(
 t;f\left(\sum_{k=1}^{n}a_k^*x_ka_k\right)
 \right),
 \quad t>0,
\]
i.e.,
\[
 \sum_{k=1}^{n}a_k^*f(x_k)a_k
 \prec\!\prec
 f\left(\sum_{k=1}^{n}a_k^*x_ka_k\right).
\]
When $f$ is strictly concave and the two operators above belong to
$L_1(\mathcal M,\tau)$, equality in the concave inequality of
Corollary~\ref{prop:two contractions} holds if and only if
\[
 \mu\left(
 t;f\left(\sum_{k=1}^{n}a_k^*x_ka_k\right)
 \right)
 =
 \mu\left(
 t;\sum_{k=1}^{n}a_k^*f(x_k)a_k
 \right),
 \quad t>0.
\]
This follows by applying to $-f$ an argument analogous to that used
above in the strictly convex case.

\end{remark}

By applying the two-variable argument in
\cite[Remark 11]{BrownKosaki} inductively to successive partial sums
and using Remark \ref{positive case} at each step,
we obtain the equality conditions for the submajorization inequalities in
\cite[Theorem 5.3]{DS}.

\begin{proposition}\label{prop:eq con submajor}
Let $x_1,\dots,x_n\in S(\cM,\tau)_+$
and let $f:[0,\infty)\to[0,\infty)$ be continuous with $f(0)=0$.
Assume that all traces below are finite.
\begin{enumerate}
\item If $f$ is convex, then
\[
\sum_{k=1}^nf(x_k)\prec\!\prec
f\left(\sum_{k=1}^n x_k\right); 
\]
\item If $f$ is concave,
then the reverse inequality holds;
\item 
In the case where $f$ is strictly convex or strictly concave, 
\[
\mu\left(\sum_{k=1}^nf(x_k)\right)=\mu\left(f\left(\sum_{k=1}^n x_k\right)\right)
\]
if and only if $x_1,\dots,x_n$ are pairwise orthogonal.
\end{enumerate}
\end{proposition}

\section{A noncommutative Lamperti-type theorem}

In this section, we extend the Lamperti-type inequality and its equality condition (see \cite[Theorem~2.1]{Lamperti}) to the noncommutative setting.
Having this at hand, we obtain a characterization of isometries on a certain class of $F$-normed noncommutative  Orlicz space 
(see Section \ref{Orlicz pre}).
This extends results in \cite{Lamperti}.

Throughout this section, unless stated otherwise, we assume that $\cM$ is a semifinite von Neumann algebra equipped with a semifinite faithful normal trace $\tau$.

\subsection{Noncommutative version of Lamperti-type inequality}

We begin by introducing two propositions which serve as the main tools in  the proof of the noncommutative Lamperti-type inequality and its equality condition.
The following proposition follows from
Corollary \ref{prop:two contractions} by taking
$n=2,\,a_1=a_2=\frac{\sqrt{2}}{2}\mathbf{1}$.

\begin{proposition}\label{lem:strict-jensen}
Let $x,y\in S(\cM,\tau)_+$ and let
$f$ be continuous on $[0,\infty)$ with $f(0)=0$.  
Assume all traces
below are finite.
If $f$ is convex, then
\[
    2\tau \left(f \left(\frac{x+y}{2}\right)\right)
    \le
    \tau(f(x))+\tau(f(y)).
\]
In the case where $f$ is strictly convex, equality holds if and only if $x=y$.

If $f$ is concave, then the reverse inequality holds. 
In the case where $f$ is strictly concave,
equality again holds if and only if $x=y$.

\end{proposition}

The following proposition is a direct consequence of Proposition \ref{prop:eq con submajor},
and extends result in \cite[Remark 11]{HaradaKosaki}.

\begin{proposition}\label{prop:rk-equality}
Let $x,y\in S(\cM,\tau)_+$ and let $f:[0,\infty)\to[0,\infty)$ be continuous with $f(0)=0$.
Assume that all traces below are finite.
\begin{enumerate}[label={\rm(\alph*)}]
\item If $f$ is convex, then
\[
\tau(f(x))+\tau(f(y))\le \tau(f(x+y)).
\]
In the case where $f$ is strictly convex, equality holds if and only if $xy=0$.
\item If $f$ is concave,
then the reverse inequality holds.
In the case where $f$ is strictly concave, equality holds if and only if $xy=0$.
\end{enumerate}
\end{proposition}

\begin{proof}[Proof of Theorem \ref{thm:lamperti}]
Let $a:=|x+y|^2$ and $b:=|x-y|^2.$
Then,
\begin{align*}
\frac{a+b}{2}=|x|^2+|y|^2.
\end{align*}

Assume first that $\psi$ is convex.  
By Proposition~\ref{lem:strict-jensen} and Proposition \ref{prop:rk-equality}, we have
\begin{equation}\label{inequal chain}
\begin{aligned}
\tau(\Phi(|x+y|))+\tau(\Phi(|x-y|))=&
\tau(\psi(a))+\tau(\psi(b))\\
\ge&
2\tau\left(\psi\left(\frac{a+b}{2}\right)\right)\\
=&2\tau\left(\psi\left(|x|^2+|y|^2\right)\right)\\
\ge&
2\tau\left(\psi\left(|x|^2\right)\right)+2\tau\left(\psi\left(|y|^2\right)\right)\\
=&2\tau(\Phi(|x|))+2\tau(\Phi(|y|)).
\end{aligned}
\end{equation}
The concave
inequality is identical, with both inequalities reversed.

It remains to prove the equality condition.  
Assume that $\psi$ is strictly convex (the strictly concave case is handled identically).
If equality holds, then 
both inequalities in \eqref{inequal chain} must be equalities.
By Proposition~\ref{lem:strict-jensen} and Proposition \ref{prop:rk-equality}, we have
\[
|x+y|^2=a=b=|x-y|^2, 
\quad
|x|^2|y|^2=0.
\] 
which implies that 
\begin{align*}
x^*y+y^*x=0, 
\quad 
s(x)s(y)=s\left(|x|^2\right)s\left(|y|^2\right)=0.
\end{align*}
Consequently, $xy^*=0$ and
\begin{align*}
    x^*y=r(x^*)(x^*y)s(y)+s(x)s(y)(y^*x)s(x)s(y)
    &=s(x)(x^*y)s(y)+s(x)(y^*x)s(y)\\
    &=s(x)(x^*y+y^*x)s(y)=0. 
\end{align*}


Conversely, if $x^*y=0=xy^*$, then
\[
a=|x+y|^2=|x-y|^2=b, 
\quad
s\left(|x|^2\right)s\left(|y|^2\right)=s(x)s(y)=0.
\]
By Proposition \ref{lem:strict-jensen} and Proposition \ref{prop:rk-equality}, both inequalities in \eqref{inequal chain} must be equalities.
\end{proof}

\begin{remark}
Taking $\Phi(t)=t^p$, we have $\psi(t)=t^{p/2}$,
which is strictly convex when $2<p<\infty$ and is strictly concave when $0<p<2$.  
In this case, Theorem \ref{thm:lamperti} recovers noncommutative Clarkson's inequality and its equality conditions
(see \cite[Theorem 1]{Yeadon} and \cite[Theorem A.1]{RX}).

Hirzallah--Kittaneh proved corresponding inequalities for
unitarily invariant norms of compact operators \cite{HKittaneh}, and
Dauitbek--Tleulessova obtained related versions in symmetric spaces of
$\tau$-measurable operators \cite{DT}.  

\end{remark}

\subsection{Characterization of isometries on $L_\Phi(\cM,\tau)$}

Throughout this subsection, unless stated otherwise,
we assume that $\Phi$ on $[0,\infty)$ is continuous and strictly increasing, with $\Phi(0)=0$.
Let $\cM$ and $\cN$ be two semifinite von Neumann algebras equipped with semifinite faithful normal traces $\tau$ and $\nu$, respectively.

Denote
\[
Y_\Phi(\cM,\tau):=\{x\in S(\cM,\tau):\tau(\Phi(|x|))<\infty\},
\]
\[
Y_\Phi(\cN,\nu):=\{x\in S(\cN,\nu):\nu(\Phi(|x|))<\infty\}.
\]

Then, we have $Y_\Phi(\cM,\tau)$ (respectively, $Y_\Phi(\cN,\nu)$) is a balanced and absorbing subset of $L_\Phi(\cM,\tau)$.
If $T:Y_\Phi(\cM,\tau)\to Y_\Phi(\cN,\nu)$ is a linear 
\footnote{
Note that 
$Y_\Phi(\cM,\tau)$ may not be a linear space. 
In this case, we say that $T$ is linear if 
\begin{enumerate}
\item $T(\lambda x)=\lambda T(x),\,\lambda\in \mathbb{C}$, $x,\lambda x\in Y_\Phi(\cM,\tau)$;
\item $T(x+y)=T(x)+T(y)$, 
$x,y,x+y\in Y_\Phi(\cM,\tau)$.
\end{enumerate}
}
mapping, 
then $T$ can be uniquely extended to a linear mapping, still denoted by  $T:L_\Phi(\cM,\tau)\to L_\Phi(\cN,\nu)$.
Indeed, 
define
\begin{align*}
    T(x):=\frac{1}{\lambda}T(\lambda x)\in L_\Phi(\cN,\nu),\quad x\in L_\Phi(\cM,\tau),
\end{align*}
where $\lambda>0$ is such that $\lambda x\in Y_\Phi(\cM,\tau)$ ($Y_\Phi(\cM,\tau)$ is absorbing).
For each $x\in L_\Phi(\cM,\tau)$,
if $0<\lambda_1<\lambda_2$ such that 
$\lambda_1x,\lambda_2x\in Y_\Phi(\cM,\tau)$,
then the linearity of $T$ on $Y_\Phi(\cM,\tau)$ implies that
\[
\frac{1}{\lambda_1}T(\lambda_1x)=
\frac{1}{\lambda_1}T\left(\frac{\lambda_1}{\lambda_2}\lambda_2x\right)
=\frac{1}{\lambda_2}T(\lambda_2x).
\]
Hence, the extension is unique.
Moreover, it is readily verified that $T(\lambda x)=\lambda T(x)$ for all $\lambda\in \mathbb{C}$.
For each $x,y\in L_\Phi(\cM,\tau)$, since $Y_\Phi(\cM,\tau)$ is balanced and absorbing, it follows that there exists a sufficiently small $\lambda>0$ such that $\lambda x,\lambda y,\lambda(x+y)\in Y_\Phi(\cM,\tau)$. 
Then,
\[
T(x+y)=\frac{1}{\lambda}T(\lambda(x+y))=
\frac{1}{\lambda}T(\lambda x)+\frac{1}{\lambda}T(\lambda y)
=T(x)+T(y).
\]
Consequently, $T$ is linear on $L_\Phi(\cM,\tau)$.
If, in addition, $T:Y_\Phi(\cM,\tau)\to Y_\Phi(\cN,\nu)$ is surjective, then
$T:L_\Phi(\cM,\tau)\to L_\Phi(\cN,\nu)$ is surjective.

Applying Lemma \ref{modular-Fnorm con}, we have the following result.

\begin{proposition}\label{mo-Fnorm}
$T|_{Y_\Phi(\cM,\tau)}$ is modular-preserving if and only if 
$T:L_\Phi(\cM,\tau)\to L_\Phi(\cN,\nu)$ is an isometry.
\end{proposition}

\begin{proof}
Assume that $T|_{Y_\Phi(\cM,\tau)}$ is modular-preserving.
For each $x\in L_\Phi(\cM,\tau)$,
since 
\[
\tau\left( \Phi\left(\frac{|x|}{\lambda }\right) \right)   = 
\nu\left( \Phi\left(\frac{|T(x)|}{\lambda }\right) \right) 
\]
for all $\lambda>0$ such that $\frac{x}{\lambda }\in Y_\Phi(\cM,\tau)$, 
it follows from the definition that 
\[
\norm{x}_\Phi =  \cNorm{T(x)}_\Phi,
\]
i.e., $T:L_\Phi(\cM,\tau)\to L_\Phi(\cN,\nu)$ is an isometry.

Conversely, assume that $T:L_\Phi(\cM,\tau)\to L_\Phi(\cN,\nu)$ is an isometry.
For each $x\in Y_\Phi(\cM,\tau)$,
by Lemma \ref{modular-Fnorm con}, we have
\begin{align*}
\tau(\Phi(|x|))
=&\inf\left\{\lambda>0:\tau\left(\Phi\left(|x|\right)\right)\le \lambda\right\}
=\inf\{\lambda>0:\cNorm{\lambda x}_\Phi\le \lambda\}\\
=&
\inf\{\lambda>0:\cNorm{\lambda T(x)}_\Phi\le \lambda\}
=\inf\{\lambda>0:\nu(\Phi(|T(x)|))\le \lambda\}\\
=&\nu(\Phi(|T(x)|)),
\end{align*}
i.e., $T|_{Y_\Phi(\cM,\tau)}$ is modular-preserving.
The proof is complete.
\end{proof}

\begin{lemma}\label{Phibp}
Assume that $\Phi$ is a Borel function with $\Phi(0)=0$.   
If $b\ge 0$ is affiliated with $\cM$ and $p\in P(\cM)$ commutes with all spectral projections of $b$,
then
\[
\Phi(bp)=\Phi(b)p.
\]
\end{lemma}

\begin{proof}
Since $pe^{b}(\delta)=e^{b}(\delta)p,\, \delta\subset [0,\infty),$
it follows from \cite[Proposition 7.7]{Haase2020} that
\[
pb\subseteq bp,\quad p\Phi(b)\subseteq\Phi(b)p,
\]
i.e.,
$p\cH$ and $(\mathbf{1}-p)\cH$ are reducing subspaces for both $b$ and $\Phi(b)$.

Let $b_p$ and $b_{\mathbf{1}-p}$ denote the parts of $b$ in $p\cH$ and $(\mathbf{1}-p)\cH$,
respectively. With respect to the orthogonal decomposition
$\cH=p\cH\oplus {(\mathbf{1}-p)}\cH,$
we have
\begin{align}\label{b direct decomposition}
    b=b_p\oplus b_{\mathbf{1}-p},
    \quad
    bp=b_p\oplus 0.
\end{align}
For all $\delta\subset(0,\infty)$, by the spectral theorem, we have $e^{b_p}(\delta)\le p$. 
Therefore,
\begin{align}\label{Phi(bp)}
\Phi(bp)
\stackrel{\eqref{b direct decomposition}}{=}
      \Phi(b_p\oplus 0) 
      \stackrel{\Phi(0)=0}{=}\Phi(b_p)\oplus 0.
\end{align}
On the other hand,
we similarly obtain that
\[
    \Phi(b)p
    \stackrel{\eqref{b direct decomposition}}{=}
    \Phi(b_p\oplus b_{\mathbf{1}-p})p
      =\left(\Phi(b_p)\oplus\Phi(b_{\mathbf{1}-p})\right)p
      =\Phi(b_p)\oplus 0,
\]
which together with \eqref{Phi(bp)} yields that
\[
\Phi(bp)=\Phi(b)p.
\]
\end{proof}

Let $\mathcal E\subset S(\cM,\tau)$ be an $\cM$-bimodule.
A net $\{x_i\}_{i\in I}\subset \mathcal E$ is said to be order convergent to $x$ (denoted by $x_i\xrightarrow{(o)}x$) whenever there exists a net $\{u_i\}_{i\in I}\subset \mathcal E_+$ such that $u_i\downarrow 0$ and $|x_i-x|\leq u_i$ for each $i \in I$.

A complex-linear mapping $J:\cM\to\cN$ 
is called a {\it Jordan $*$-homomorphism} if 
\[
J(x^2) = J(x)^2  \text{ and }  J(x^*) = J(x)^* ,\, x \in \cM.
\]
We call $J$ a \textit{Jordan $*$-monomorphism} if it is injective,
and a \textit{Jordan $*$-isomorphism} if it is bijective.
A Jordan $*$-isomorphism $J$ is necessarily normal \cite[Appendix A]{RR}.
Further details regarding Jordan homomorphisms may be
found in \cite{KadisonRingroseI, BraRobinbook}.






\begin{proof}[Proof of Theorem \ref{Tform}]

Let $x,y\in L_\Phi(\cM,\tau)$ be such that $x^*y=0=xy^*$.
To prove $T$ is disjointness preserving,
without loss of generality, we may assume that $x,y\in Y_\Phi(\cM,\tau)$ ($Y_\Phi(\cM,\tau)$ is absorbing).
Note that $T$ is modular-preserving (see Proposition \ref{mo-Fnorm}). 
By Theorem \ref{traceineq}, we have 
\begin{align*}
\nu\left(\Phi\left(|T(x+y)|\right)\right)+\nu\left(\Phi\left(|T(x-y)|\right)\right)
=&
\tau(\Phi(|x+y|))+\tau(\Phi(|x-y|))\\
=&2\tau(\Phi(|x|))+2\tau(\Phi(|y|))\\
=&
2\nu\left(\Phi\left(|T(x)|\right)\right)+2\nu\left(\Phi\left(|T(y)|\right)\right),
\end{align*}
which together with Theorem \ref{traceineq} implies that $T(x)^*T(y)=0=T(x)T(y)^*$. Consequently, $T$ is disjointness-preserving.


We claim that $T$ is order-measure continuous.
Let $x\in L_\Phi(\cM,\tau)$ and $\{x_i\}_{i\in I}\subset L_\Phi(\cM,\tau)$ be such that $x_i\xrightarrow{(o)}x$.
Then, there exists $\{u_i\}_{i\in I}\subset L_\Phi(\cM,\tau)_+$ satisfying $|x-x_i|\le u_i\downarrow 0$. 
Since $\{u_i\}_{i\in I}\subset S_0(\cM,\tau)$ (see Section \ref{Orlicz pre}), it follows from \cite[Theorem 2.6.3]{DPS} that $u_i\xrightarrow{t_m}0$, which together with \eqref{m-mu} and \cite[Proposition 3.2.7(v)]{DPS} yields that
\begin{align}\label{muui}
\mu(t;u_i)\downarrow_i 0,
    \quad t>0.
\end{align}
Fixed $i_{0}\in I$, 
there exists
$\lambda>0$ 
such that $\lambda u_{i_0}\in Y_\Phi(\cM,\tau)$.
Since $\Phi$ is strictly increasing and continuous with $\Phi(0)=0$, 
it follows from \cite[Proposition 3.2.8]{DPS} that 
\begin{align}\label{mu(Phi ui)}
\mu(t;\Phi(\lambda u_i))
=\Phi(\lambda\mu(t;u_i))\stackrel{\eqref{muui}}{\downarrow_i} 0,
\quad t> 0,
\end{align} 
i.e., $\Phi(\lambda u_i)\xrightarrow{t_m}0$
(see \eqref{m-mu}).
Note that \[
\tau(\Phi(\lambda u_i))\to 0.
\]
Indeed, 
assume that
$\tau(\Phi(\lambda u_i))\not\to 0.$
Then, there exist $\varepsilon_0>0$ and a sequence 
$\{u_{i_n}\}\subset \{u_i\}_{i\ge i_0}$ 
such that 
$$\Phi\left(\lambda u_{i_n}\right)\xrightarrow{t_m}_n 0
\quad \text{and}\quad
\tau(\Phi(\lambda u_{i_n}))\ge \varepsilon_0.$$
Since $\lambda u_{i_0}\in Y_\Phi(\cM,\tau)$, 
it follows from \eqref{mu(Phi ui)} that 
$\mu(\Phi(\lambda u_{i_n}))\le \mu(\Phi(\lambda u_{i_0}))\in L_1(0,\infty)$.
By \cite[Theorem 3.4.21]{DPS}, we have $\tau(\Phi(\lambda u_{i_n}))\to 0$, which is a contradiction.
Therefore, 
\[
\tau(\Phi(\lambda u_i))\to 0.
\]
Since $\Phi$ is strictly increasing and continuous with $\Phi(0)=0$,
it follows from \cite[Proposition~3.2.7(v) and Proposition~3.2.8]{DPS} that
\begin{align}\label{mu(t) inequality}
\mu(\Phi\left(\lambda|x-x_i|\right))=\Phi(\mu(\lambda|x-x_i|))
\le \Phi(\mu(\lambda u_i))=\mu(\Phi(\lambda u_i)),
\quad i\ge i_0.
\end{align}
Recall that $T$ preserves modular. Then,
\[
 \nu\left(\Phi\left(\lambda\left|T(x-x_i)\right|\right)\right)
     =\nu(\Phi(|T(\lambda(x-x_i))|))
     =
    \tau\left(\Phi\left(\lambda|x-x_i|\right)\right) 
\stackrel{\eqref{mu(t) inequality}}{\le}
    \tau\left(\Phi(\lambda u_i)\right)\to 0.
\]
By \cite[Proposition 3.2.7(i) and Proposition 3.3.9]{DPS}, we have
$\lambda\mu\left(t;|T(x-x_i)|\right)=\mu\left(t;\lambda|T(x-x_i)|\right)\to_i 0$ for all $t>0$,
i.e., $$|T(x-x_i)|\xrightarrow{t_m}0$$
(see \eqref{m-mu}), which proves our claim.

By \cite[Theorem 3.5]{FHKX},
we have 
\begin{align}\label{T:bounded}
T(x)=ubJ(x),\quad x\in L_\Phi(\cM,\tau)\cap \cM,
\end{align}
where $u\in \cN$ is a partial isometry,
$b$ is a (possibly not measurable) positive operator affiliated with $\cN$ and $J:\cM\to \cN$ is a (normal) Jordan $*$-monomorphism.
In particular, 
\[
u^*u=s(b)=J(\mathbf{1}),
\]
\begin{align}\label{eb}
e^b(\delta)=
so-\lim\limits_{p\in P(\cM)\cap\cF(\cM,\tau)}e^{|T(p)|}(\delta)\in Z(J(\cM)),
\quad \delta\subset \mathbb{R},
\end{align}
and $u^*T(\cdot)$
is normal (see the proof of \cite[Theorem 3.5]{FHKX}, or \cite{HSZ20}).

For each $\tau$-finite $e\in P(\cM)$, 
by Proposition \ref{mo-Fnorm} and Lemma \ref{Phibp}, 
we obtain that
\[
\tau(e)
\stackrel{\Phi(1)=1}{=}
\tau(\Phi(e))=\nu(\Phi(|T(e)|))
\stackrel{\eqref{T:bounded}}{=}
\nu(\Phi(bJ(e)))
\stackrel{\text{Lemma~\ref{Phibp}}}{=}
\nu(\Phi(b)J(e)).\]
By uniform approximation and Lemma \ref{Phibp}, we obtain
\begin{align}\label{traceeq Ftau}
\tau(x)=\nu(\Phi(b)J(x))=\nu(\Phi(bJ(s(x)))J(x))
\quad 0\le x\in \cF(\cM,\tau).
\end{align}

For each $0\le x\in L_1(\cM,\tau)\cap \cM$, 
since $x\in S_0(\cM,\tau)$, 
it follows that
\[
e_n:=e^x\left(\frac{1}{n},\infty\right)\in \cF(\cM,\tau),\quad 0<n\in \mathbb{N}.
\]
Then, we have $J(e_n)\uparrow J(s(x))$ ($J$ is normal).
Furthermore, by 
Lemma \ref{Phibp}, we have
\begin{equation}\label{sup nu}
\nu(\Phi(b)J(xe_n))=\nu(\Phi(bJ(e_n))J(xe_n))
\stackrel{\eqref{traceeq Ftau}}{=}
\tau(xe_n)\uparrow\tau(x).
\end{equation}

Next, we prove that $\Phi(b)J(xe_n)\uparrow\Phi(b)J(x)$.
By \eqref{eb} and \cite[Proposition 5.27 and Theorem 5.23]{Schmudgen}, 
there exists a unique spectral measure $E$ on $\cB(\mathbb{R}^2)$ such that
\begin{align*}
\Phi(b)=\int_{\mathbb{R}^2}t_1dE(t_1,t_2),
\quad
J(x)=\int_{\mathbb{R}^2}t_2dE(t_1,t_2).
\end{align*}
Noting that $J(x)$ is bounded and $\Phi(b)$ is closed,
we have $\Phi(b)J(x)$ is closed, 
which together with \cite[Theorem 2.4(b)]{Haase2020} yields that
\begin{align}\label{PhibJ(x) form}
\Phi(b)J(x)=
\overline{\int_{\mathbb{R}^2}t_1dE(t_1,t_2)\int_{\mathbb{R}^2}t_2dE(t_1,t_2)}=\int_{\mathbb{R}^2}t_1t_2dE(t_1,t_2)\ge 0.
\end{align}
By \cite[Proposition 2.9.2(iv)]{DPS}, we have
\[
J(e_n)=\chi_{(1/n,\infty)}(J(x))=\int_{\mathbb{R}^2}\chi_{\{\mathbb{R}\times(1/n,\infty)\}}dE(t_1,t_2).
\]
Similarly to \eqref{PhibJ(x) form}, we obtain
\begin{align*}
\Phi(b)J(xe_n)=\Phi(b)J(x)J(e_n)=
\int_{\mathbb{R}^2}t_1t_2\chi_{(1/n,\infty)}(t_2)dE(t_1,t_2).
\end{align*}
For each $\xi\in \cD\left((\Phi(b)J(x))^\frac{1}{2}\right)$,
we have
\begin{align*}
\norm{(\Phi(b)J(xe_n))^\frac{1}{2}\xi}_{\cH}^2=&
\langle\Phi(b)J(xe_n)\xi,\xi\rangle=
\int_{\mathbb{R}^2}t_1t_2\chi_{(1/n,\infty)}(t_2)
d\langle E(t_1,t_2)\xi,\xi\rangle\\
\uparrow_n&
\int_{\mathbb{R}^2}t_1t_2d\langle E(t_1,t_2)\xi,\xi\rangle
=\langle\Phi(b)J(x)\xi,\xi\rangle\\=&
\norm{(\Phi(b)J(x))^\frac{1}{2}\xi}_{\cH}^2<\infty.
\end{align*}
Moreover, this implies
\[
\cD\left((\Phi(b)J(x))^\frac{1}{2}\right)
=\left\{\xi\in \bigcap_n\cD\left((\Phi(b)J(xe_n))^\frac{1}{2}\right)
:\sup_n\norm{(\Phi(b)J(xe_n))^\frac{1}{2}\xi}_{\cH}^2<\infty\right\}.
\]
Hence, $\Phi(b)J(xe_n)\uparrow\Phi(b)J(x)$.
This together with \eqref{sup nu} and \cite[Theorem~3.4.12]{DPS}
implies that
\[
\nu(\Phi(b)J(x))=\sup_n\nu(\Phi(b)J(x e _n))\stackrel{\eqref{sup nu}}{=}\tau(x).
\]
The proof is complete.

\end{proof}

\begin{remark}
\leavevmode\par
\begin{enumerate}
\item Proposition \ref{mo-Fnorm} shows that Theorem \ref{Tform}  extends the description in \cite[Theorem 4.1]{Lamperti} to the noncommutative setting.
\item When $\Phi(t)=t^p$, $0<p<\infty$,$p\ne 2$, we have $\rho_\Phi(x)
=\tau(|x|^p)
=\norm{x}_p^p$.
 Theorem~\ref{Tform} provides a characterization of (not necessarily surjective) isometries on $L_p(\cM,\tau)$, which extends the result in \cite[Theorem 2]{Yeadon}.
\item The condition $\Phi(1)=1$ in Theorem \ref{Tform} serves only for normalization purposes and does not alter the general form of isometries.
If $\Phi(1)\ne 1$, then the trace relation takes the form
\[
\Phi(1)\tau(x)=\nu(\Phi(b)J(x)), \quad 0\le x\in L_1(\cM,\tau)\cap \cM.
\]
\end{enumerate}
\end{remark}


\begin{thebibliography}{99}





\bibitem{BGL2022}
D. Blecher, S. Goldstein, L. Labuschagne,
{\it Abelian von Neumann algebras, measure algebras and
\(L^\infty\)-spaces,}
Expo. Math. \textbf{40} (2022), no.~3, 758--818.


\bibitem{Bogachev2007}
V. Bogachev, {\it Measure Theory,} Vol. I, Springer--Verlag, Berlin, 2007.

\bibitem{BraRobinbook}
O. Bratteli, D. Robinson,
{\it Operator algebras and quantum statistical mechanics I,}
2nd ed., Springer-Verlag, New York, 1987.


\bibitem{BrownKosaki}
L.  Brown, H.~Kosaki,
{\it Jensen's inequality in semi-finite von Neumann algebras},
J. Operator Theory \textbf{23} (1990), no.~1, 3--19.


\bibitem{CarlenFrankLarson2025}
E.  Carlen, R. Frank, S. Larson,
{\it A Jensen inequality for partial traces and applications to partially
semiclassical limits},
Lett. Math. Phys. \textbf{115} (2025), no. 3, 15 pp.

\bibitem{Choi1974}
M. Choi,
{\it A Schwarz inequality for positive linear maps on $C^{*}$-algebras},
Illinois J. Math. 18 (1974), no. 4, 565--574.



\bibitem{DPS}
P. Dodds, B. de Pagter, F. Sukochev,
{\it Noncommutative integration and operator theory},
Progress in Mathematics, Birkhäuser, Cham, 2024.


\bibitem{DS}
P. Dodds and F. Sukochev,
{\it Submajorisation inequalities for convex and concave functions of sums of measurable operators},
Positivity \textbf{13} (2009), no.~1, 107--124.

\bibitem{DT}
D.~Dauitbek, A.~M. Tleulessova,
{\it Non-commutative Clarkson inequalities for symmetric space norm of
$\tau$-measurable operators},
Int. J. Math. Anal. (Ruse) \textbf{7} (2013), no.~18,
883--890.



\bibitem{Davis1957}
C. Davis,
{\it A Schwarz inequality for convex operator functions},
Proc. Amer. Math. Soc. 8 (1957), no. 1, 42--44.



\bibitem{FackKosaki}
T.~Fack,  H.~Kosaki,
{\it Generalized $s$-numbers of $\tau$-measurable operators},
Pacific J. Math. \textbf{123} (1986), no.~2, 269--300.


\bibitem{FHKX}
K. Fang, J. Huang, K. Kudaybergenov, R. Xu,
{\it Disjointness-preserving mappings on Calkin operator spaces and positive isometries},
arXiv:2607.26563, 2026.

\bibitem{Haase2020}
M.~Haase,
{\it The functional calculus approach to the spectral theorem},
Indag. Math. (N.S.),
\textbf{31} (2020), no.~6, 1066--1098.



\bibitem{HansenPedersen1982}
F. Hansen, G.  Pedersen,
{\it Jensen's inequality for operators and L\"owner's theorem},
Math. Ann. 258 (1982), no. 3, 229--241.




\bibitem{HansenPedersen2003}
F. Hansen, G. K. Pedersen,
{\it Jensen's operator inequality},
Bull. Lond. Math. Soc. 35 (2003), no. 4, 553--564.





\bibitem{HaradaKosaki}
T. Harada, H. Kosaki,
{\it On equality condition for trace Jensen inequality in semi-finite von
Neumann algebras},
Internat. J. Math.  \textbf{19} (2008), no.~4, 481--501.


\bibitem{HaradaKosaki2010}
T. Harada, H. Kosaki,
{\it Trace Jensen inequality and related weak majorization in semi-finite
von Neumann algebras},
J. Operator Theory 63 (2010), no.~1, 129--150.






\bibitem{HKittaneh}
O.~Hirzallah, F.~Kittaneh,
{\it Non-commutative Clarkson inequalities for unitarily invariant norms},
Pacific J. Math. \textbf{202} (2002), no.~2, 363--369.


\bibitem{HLS}
J. Huang, G. Levitina, F. Sukochev,  
{\it Completeness of symmetric $\Delta$-normed spaces of $\tau$-measurable operators}, Studia Math. 237 (2017), no. 3, 201--219.


\bibitem{HSZ20}
J. Huang, F. Sukochev, D. Zanin,
{\it Logarithmic submajorisation and order-preserving linear operators},
J. Funct. Anal. 278 (2020), no. 4, 108352, 44 pp.



\bibitem{Jensen1906}
J. Jensen,
{\it Sur les fonctions convexes et les in{\'e}galit{\'e}s entre les valeurs moyennes},
Acta Math. 30 (1906), 175--193.











\bibitem{KadisonRingroseI}
R. Kadison, J. Ringrose,
{\it Fundamentals of the Theory of Operator Algebras,
Volume I: Elementary Theory},
Graduate Studies in Mathematics 15,
American Mathematical Society, Providence, RI, 1997.



\bibitem{Kallenberg2021}
O. Kallenberg, 
{\it Foundations of Modern Probability,} 
3rd ed., Probability Theory and Stochastic Modelling 99, Springer, Cham, 2021.



\bibitem{Kosaki2013}
H. Kosaki,
{\it Trace Jensen inequality for self-adjoint operators in semi-finite
von Neumann algebras},
Internat. J. Math. 24 (2013), no.~9,
Paper No. 1350075, 15 pp.




\bibitem{Lamperti}
J.~Lamperti,
{\it On the isometries of certain function-spaces},
Pacific J. Math. \textbf{8} (1958), no.~3, 459--466.







\bibitem{LSZ}
S. Lord, F. Sukochev, D. Zanin,
{\it Singular traces: Theory and applications},
De Gruyter Studies in Mathematics 46, De Gruyter, Berlin, 2013.


\bibitem{MazurOrlicz}
S. Mazur, W. Orlicz, {\it On some classes of linear spaces.}
Studia Math. 17, (1959), 97--117.



\bibitem{MO}
J. Musielak, W. Orlicz,  {\it On modular spaces}, 
Studia Math. 18 (1959), no.~1, 49--65.








\bibitem{Petz1986}
D. Petz, \textit{On the equality in Jensen's inequality for operator
convex functions}, Integral Equations Operator Theory \textbf{9}
(1986), no.~5, 744--747.

\bibitem{Petz1987}
D. Petz,
{\it Jensen's inequality for positive contractions on operator algebras},
Proc. Amer. Math. Soc. 99 (1987), no. 2, 273--277.



\bibitem{RahamanTurowska2026}
M. Rahaman, L. Turowska,
{\it Jensen's inequality for partial traces in von Neumann algebras},
Bull. Lond. Math. Soc. \textbf{58} (2026), no. 4, 12 pp.


\bibitem{RX}
Y. Raynaud, Q. Xu, {\it On subspaces of non-commutative $L_p$-spaces}, J. Funct. Anal. \textbf{203} (2003), no. 1, 149--196.




\bibitem{RR}
A. Rieckers, H. Roos,
{\it Implementation of Jordan-isomorphisms for general von Neumann algebras,}
Ann. Inst. H. Poincar\'e Phys. Th\'eor. 50 (1989), no. 1, 95--113.


\bibitem{Rudin}
W. Rudin, \textit{Real and Complex Analysis},
3rd ed., McGraw-Hill, New York, 1987.



\bibitem{Schmudgen}
K. Schm\"udgen, \textit{Unbounded Self-adjoint Operators on Hilbert Space},
Graduate Texts in Mathematics 265, Springer, Dordrecht, 2012.







\bibitem{Takesaki}
M. Takesaki, \textit{Theory of Operator Algebras II},
Encyclopaedia of Mathematical Sciences 125,
Springer-Verlag, Berlin, 2003.

\bibitem{Umegaki1956}
H.~Umegaki,
{\it Conditional expectation in an operator algebra, II,}
Tohoku Math. J. (2) \textbf{8} (1956), no. 1, 86--100.



\bibitem{weigt}
M. Weigt,
{\it Jordan homomorphisms between algebras of measurable operators},
Quaest. Math. 32 (2009), no. 2, 203--214.



\bibitem{Yeadon}
F. Yeadon,
{\it Isometries of non-commutative $L^{p}$-spaces},
Math. Proc. Cambridge Philos. Soc. 90 (1981), no. 1,  41--50.


\end{thebibliography}
\end{document}